\documentclass[12pt,a4paper]{article}
\usepackage{amsmath}
\usepackage{amsthm}
\usepackage{amssymb}
\usepackage{amsbsy}
\usepackage{amsfonts}
\usepackage{amstext}
\usepackage{amscd}
\usepackage[dvips]{graphicx,psfrag}
\numberwithin{equation}{section}
\theoremstyle{plain}
\newtheorem{theorem}{Theorem}[section]
\newtheorem{proposition}[theorem]{Proposition}
\newtheorem{corollary}[theorem]{Corollary}
\newtheorem{lemma}[theorem]{Lemma}
\theoremstyle{definition}
\newtheorem{example}[theorem]{Example}
\newtheorem{remark}[theorem]{Remark}
\newtheorem{definition}[theorem]{Definition}

\newcommand{\disc}{\mathbb{D}}

\newcommand{\qa}{\mathcal{Q}}

\newcommand{\pa}{\mathcal{P}}
\newcommand{\pc}{\mathcal{P}_c}
\newcommand{\pmo}{\mathcal{P}_m}
\newcommand{\ps}{\mathcal{P}_{sym}}
\DeclareMathOperator*{\p+}{\mathcal{P_{+}}}
\newcommand{\pb}{\mathcal{P}^2}
\newcommand{\real}{\mathbb{R}}
\newcommand{\tor}{\mathbb{T}}
\DeclareMathOperator*{\com+}{\mathbb{C}_+}
\newcommand{\comp}{\mathbb{C}}
\newcommand{\nat}{\mathbb{N}}
\newcommand{\im}{\text{Im}}
\newcommand{\re}{\text{Re}}
\newcommand{\supp}{\text{supp}\,}

\newcommand{\rt}{{\kern0.2em \rhd_{\displaystyle\scriptscriptstyle T}}}
\newcommand{\pn}{\mathcal{P}(n)}
\newcommand{\ncpn}{\mathcal{NC}(n)}

\newcommand{\mpn}{\mathcal{M}(n)}

\newcommand{\lncpn}{\mathcal{LNC}(n)}

\newcommand{\mrhd}{\kern0.3em\rule{0.1mm}{0.52em}\kern-.35em\gtrdot}

\begin{document}
\title{Conditionally monotone independence I: Independence, additive convolutions and related convolutions}
\author{Takahiro Hasebe\footnote{This work was supported by Grant-in-Aid for JSPS Fellows.} \\ Graduate School of Science, Kyoto University,  \\ Kyoto 606-8502, Japan \\ 
Email: hsb@kurims.kyoto-u.ac.jp}
\date{}

\maketitle

\begin{abstract}
We define a product of algebraic probability spaces equipped with two states. 
This product is called a conditionally monotone product. This product is a new 
example of independence in non-commutative probability theory and unifies the monotone and Boolean products, and moreover, the orthogonal product. Then we define the associated cumulants and calculate the limit distributions in central limit theorem and Poisson's law of small numbers. We also prove a combinatorial moment-cumulant formula using monotone partitions. 
We investigate some other topics such as infinite divisibility for the additive convolution and deformations of the monotone convolution.  We define cumulants for a 
general  convolution  to analyze the  deformed convolutions.  
\end{abstract}



\section{Introduction} 
Non-commutative probability theory lays the foundation of quantum mechanics and has many mathematical branches. The basic framework consists 
of a (unital) $\ast$-algebra $\mathcal{A}$ and a state $\varphi$ on it.  The pair $(\mathcal{A}, \varphi)$ is called an algebraic probability space 
or a non-commutative probability space.   When $\mathcal{A}$ has structure of a $C^\ast$-algebra (resp. von Neumann algebra), we call the pair 
a $C^\ast$- (resp. von Neumann) algebraic probability space.  

Many kinds of independence have been studied as an aspect of non-commutative 
probability theory. The usual independence in probability theory is called tensor independence from the non-commutative probabilistic viewpoint. 
Other famous ones are free independence defined by Voiculescu \cite{V1}, Boolean independence by Speicher and Woroudi \cite{S-W} and monotone independence by Muraki 
 \cite{Mur3}. These kinds of independence can canonically be realized by using products of states on the free product of algebras (with or without the identification of units): 
there are two ``universal products'' (tensor and free) defined on the free product of algebras with the identification of units \cite{B-S1,Spe1}; there are three universal products (tensor, free and Boolean) defined on the free product of algebras without the identification of units; there are five ``natural products'' (tensor, free, Boolean, monotone and anti-monotone) defined on the free product of algebras without the identification of units \cite{Mur4,Mur5}. These results can also be understood in terms of tensor structures with inclusions \cite{BFGKT}. 
In particular, monotone and Boolean products are important in this paper. 
Moreover, the conditionally (c- for short) free product of states was introduced by Bo\.zejko, Leinert and Speicher \cite{BLS,BS}.  
This product can be seen as a universal product of pairs of states which can be defined similarly to the single state case. As such a concept 
has not been defined in the literature, we shall systematically study it elsewhere. 

In \cite{BLS} it was proved that the c-free product and cumulants unify the free and Boolean products and their cumulants introduced in \cite{Spe2,S-W,V2}.  
In addition, the c-free product also unifies the monotone product as proved in \cite{Fra}. While the latter is quite nontrivial, some complication appears in its application: 
it is difficult to repeat the calculation of monotone products in terms of c-free products; it is difficult to identify monotone cumulants \cite{H-S} in terms of c-free cumulants. 
The latter difficulty is essentially the same as the former. The solution of these difficulties is a purpose of this paper. 

To this end, in Section \ref{Condi} we introduce a c-monotone product analogously to the c-free product. Once it is introduced, the monotone and Boolean products can be formulated in terms of it. The concept of c-monotone independence can also be extracted from the c-monotone product since the product is associative. In terms of probability measures, we can also define (additive) c-monotone convolutions. We prove that c-monotone independence and c-free one include orthogonal independence \cite{Len1} as special cases. Therefore, the additive c-monotone convolution unifies the additive monotone, Boolean and orthogonal convolutions. As a result, c-monotone convolutions can give the characterization of orthogonal convolutions (Theorem 6.2 of \cite{Len1}).

 In Section \ref{Cumulants112} we introduce c-monotone cumulants $r_n(\mu, \nu)$ to linearize powers of probability measures. 
A moment-cumulant formula is proved by using combinatorics of monotone partitions; this formula is naturally expected from the monotone case \cite{H-S}. 
An important point is that c-monotone cumulants generalize monotone and Boolean cumulants. Here we achieve a purpose of this paper.  

The remaining sections are roughly divided into two parts; one is devoted to infinitely divisible distributions, and the other is to deformations of the monotone convolution.

In Sections \ref{Inf1}--\ref{Inf4}, we investigate convolution semigroups and infinitely divisible distributions. Additive monotone and Boolean infinitely divisible distributions were first studied in \cite{Mur3} and \cite{S-W}, respectively. The results in this paper generalize these studies:   
we prove the L\'{e}vy-Khintchine formula, and the correspondence among a convolution semigroup, an infinitely divisible distribution, a pair of  vector fields and a positive definite sequence of cumulants. 

Moreover, we construct convolution semigroups from monotone and Boolean ones. As a result, c-monotone cumulants $r_n(\cdot, \nu)$, for a fixed $\nu$, turn out to linearize the Boolean convolution.  We note that infinite divisibility was introduced and studied in \cite{Kry1} for c-free independence. Results on c-monotone infinite divisibility however do not follow from  the c-free case.

In Sections \ref{deform11}--\ref{Lim2} we work on deformations of the monotone convolution. This topic may impress the reader 
as specialized at first sight; however this clarifies how the structure of c-monotone independence behaves analogously to that of c-free independence.    A deformation of the free convolution can be defined in a graph of probability measures \cite{BW1,BW2,KW,KY,O1,O2}. More precisely, if $T$ is a map from the set of probability measures to itself, we can define the graph $\{(\mu, T\mu);\mu \text{~is a probability measure} \}$. If this graph is closed under the c-free convolution, we can define an associative convolution. For the details, the reader is referred to Section \ref{deform11}. Analogously, a deformed convolution arises from the c-monotone convolution of a graph of probability measures. These kinds of convolutions include the monotone and Boolean convolutions. We show many examples of such deformed convolutions. A remarkable point is that such maps $T$, found in the context of c-free convolutions, give associative convolutions also in the c-monotone case. 

The Boolean and monotone convolutions preserve the sets $\{\mu; \supp \mu \subset [0, \infty) \}$ and $\{\mu; \mu \text{ is symmetric} \}$. The former property can be proved easily in terms of the operator-theoretic approach in \cite{Fra3}; the latter can be proved by using the complex-analytic characterizations of the convolutions. As an extension of these properties, we give necessary and sufficient conditions under which the deformed convolution explained above preserves the two sets.  

We introduce the cumulants for the convolution deformed by a map $T$ and  then limit distributions are calculated for some class of such convolutions. 
When we introduce the cumulants of the deformed convolutions, the axiom of homogeneity for cumulants does not hold in general (see (\ref{C2''})). For this reason we consider the uniqueness and the existence of cumulants of a general convolution in Section \ref{Cum}.

Let us mention a few topics which are not covered in this paper.  Multiplicative convolutions and the infinite divisibility were studied in \cite{Ber1,Ber2,Fra,Fra2,Len2} in the Boolean, monotone and orthogonal cases. Multiplicative c-monotone convolutions can be similarly defined to 
generalize the monotone, Boolean and orthogonal convolutions.  We do not treat these in this paper; these aspects will be studied in \cite{Has4}.

\section{Preliminaries}
\subsection{Reciprocal Cauchy transform}
We use the notation $\com+:=\{z \in \comp; \im z > 0 \}$. 
The Cauchy transform of a probability measure $\mu$ is defined by 
\begin{equation}
G_\mu (z) = \int_{\real} \frac{1}{z-x} d\mu(x), ~z \in \comp \backslash \real. 
\end{equation}
The reciprocal Cauchy transform of a probability measure $\mu$ is defined by 
\begin{equation}
H_\mu (z) = \frac{1}{G_\mu (z)}, ~z \in \comp \backslash \real.  
\end{equation}
This is an analytic map from $\com+$ to $\com+$. Since $\lim_{y \to \infty}iyG_\mu(iy) =1$, $H_\mu$ has the following form: 
\begin{equation}\label{nev}
H_\mu (z) = b + z + \int_{\real} \frac{1 + xz}{x-z} d\eta(x),
\end{equation}
where $b \in \real$ and $\eta$ is a positive finite measure. Conversely, 
any function of the form of the right hand side of (\ref{nev}) is a reciprocal Cauchy transform of a probability measure (see \cite{Akh,Maa} for details).

\subsection{Monotone independence}
Muraki defined the concept of monotone independence in \cite{Mur3}. A definition is as follows. 
Let $(\mathcal{A}, \varphi)$ be an algebraic probability space and let $I$ be a linearly ordered set. A family of subalgebras $\{\mathcal{A}_i \}_{i \in I}$ is said to be monotone independent if the equality
\begin{equation}
\varphi (a_1 a_2 \cdots a_n)  = \varphi (a_j) \varphi (a_1a_2 \cdots a_{j-1} a_{j+1} \cdots a_n)
\end{equation} 
holds for $a_k \in \mathcal{A}_{i_k}$ with $ i_1, i_2, \cdots, i_n \in I$, $i_{j-1} < i_j > i_{j+1}$ and $1 \leq j \leq n$. When $j=1$ (resp. $j=n$), the condition $i_{j-1} < i_j > i_{j+1}$ 
is understood to be $i_1 > i_2$ (resp. $i_{n-1} < i_n$). 
The monotone convolution $\mu \rhd \nu$ is defined for two probability measures $\mu$, $\nu$ and is characterized by the formula 
\begin{equation}\label{mur}
H_{\mu \rhd \nu} = H_{\mu} \circ H_{\nu}. 
\end{equation}
The monotone convolution is non-commutative and associative. 

Let $\ncpn$ be the set of all non-crossing partitions \cite{N-S1}.  
Let $\mpn$ be the set of all \textit{monotone partitions} defined by 
\begin{equation}
\mpn := \{(\pi, \lambda): \pi \in \ncpn,~ \text{if $V, W \in \pi$ and $V$ is in the inner side of $W$, then $V >_{\lambda} W$} \},  
\end{equation}
where $\lambda $ denotes a linear ordering of the blocks of $\pi$. $V >_{\lambda} W$ means that $V$ is larger than $W$ under the linear ordering $\lambda$ (see \cite{Mur4,Mur5}, and also \cite{Len3,Len4}).  

In the paper \cite{H-S} the concept of monotone cumulants has been defined. Monotone cumulants do not satisfy 
the additivity for general probability measures, but satisfy the power additivity: $r_n(\mu ^{\rhd N}) = Nr_n(\mu)$. 
The moment-cumulant formula is described as 
\begin{equation}
m_n(\mu) = \sum_{(\pi, \lambda) \in \mpn} \frac{1}{|\pi|!}\prod_{V \in \pi} r_{|V|}(\mu). 
\end{equation}
\begin{example} 
We exhibit the moment-cumulant formula until the forth order. 
\begin{equation*}
\begin{split}
& m_1(\mu) = r_1(\mu), \\
& m_2(\mu) = r_2(\mu) + r_1 (\mu)^2, \\
& m_3(\mu) = r_3(\mu) + \frac{5}{2}r_1(\mu) r_2(\mu) + r_1(\mu) ^3, \\
& m_4(\mu) = r_4(\mu) + 3r_1 (\mu)r_3 (\mu)+ \frac{3}{2}r_2(\mu) ^2 + \frac{13}{3}r_1(\mu) ^2 r_2 (\mu)+ r_1(\mu) ^4.  
\end{split}
\end{equation*}
\end{example}

\subsection{Conditionally free independence} Let $I$ be an index set. 
Let $\mathcal{A}_i$ be a unital $*$-algebra and let $\varphi_i$, $\psi_i$ be states on $\mathcal{A}_i$ for $i \in I$. 
The c-free product of triples $(\mathcal{A}_i, \varphi_i, \psi_i)_{i \in I}$ was defined by Bo\.{z}ejko and Speicher in \cite{BS}. 
We define $(\mathcal{A}, \varphi, \psi) = \ast_{i \in I}(\mathcal{A}_i, \varphi_i, \psi_i)$ by setting $\mathcal{A}:= \ast_{i \in I}\mathcal{A}_i$ (the free product with the identification of units) and 
$\psi := \ast_{i \in I}\psi_i$ (the free product of states). $\varphi$ is defined by the following condition: the equality
\begin{equation}
\varphi (a_1 \cdots a_n) = \prod_{k = 1}^{n} \varphi_{i_k}(a_k)
\end{equation}
holds if $a_k \in \mathcal{A}_{i_k}$ with $i_1 \neq \cdots \neq i_n$ and $\psi_{i_k}(a_k) = 0$ for all $1 \leq k \leq n$.    
If the index set $I$ consists of two elements, that is, $|I| = 2$, $\varphi$ is denoted by $\varphi_1 {}_{\psi_1} \!\!\ast_{\psi_2}\!\! \varphi_2$. 

Let $\mu,~\nu$ be probability measures on $\real$ with compact supports. 
Define the $R$-transform of $\nu$ and the c-free $R$-transform of $(\mu, \nu)$ by  
\begin{gather}
\frac{1}{G_{\nu}(z)} = z - R_{\nu}(G_{\nu}(z)), \label{rel1}\\
\frac{1}{G_{\mu}(z)} = z - R_{(\mu, \nu)}(G_{\nu}(z)). \label{rel2}  
\end{gather}
We expand $R_{(\mu, \nu)}(z) = \sum_{n = 1} ^{\infty}R_n(\mu, \nu)z^{n-1}$ as a formal power series. $R_n(\mu, \nu)$ are called c-free cumulants. Similarly, we expand $R_\nu(z) =  \sum_{n = 1} ^{\infty}R_n(\nu)z^{n-1}$ and 
$R_n(\nu)$ are called free cumulants. In this paper, $H_\mu(z)$ is more useful than $G_\mu(z)$, and correspondingly, we use $\phi_{(\mu, \nu)}(z):=R_{(\mu, \nu)}(\frac{1}{z})$ and $\phi_{\nu}(z):=R_\nu(\frac{1}{z})$. Then (\ref{rel1}) and (\ref{rel2}) can be written as 
\begin{gather}
H_{\nu}(z) = z - \phi_{\nu}(H_{\nu}(z)), \label{rel13}\\
H_{\mu}(z) = z - \phi_{(\mu, \nu)}(H_{\nu}(z)) \label{re14}.   
\end{gather}

A moment-cumulant formula for c-free independence is 
\begin{equation}
m_n(\mu) = \sum_{\pi \in \ncpn} \Big{(} \prod_{\substack{V \in \pi, \\ V: \text{~outer}}}R_{|V|}(\mu, \nu) \Big{)} \Big{(}  \prod_{\substack{V \in \pi, \\ V: \text{~inner}}}R_{|V|}(\nu) \Big{)}, 
\end{equation}
which generalizes the free and Boolean moment-cumulant formulae. 

The c-free convolution of  $(\mu_1, \nu_1)$ and $(\mu_2, \nu_2)$ is the pair 
$(\mu, \nu) = (\mu_1, \nu_1) \boxplus (\mu_2, \nu_2)$, where $\mu$ and $\nu$ are characterized by 
\begin{gather}
\phi_\nu(z) = \phi_{\nu_1}(z) + \phi_{\nu_2} (z), \\
\phi_{(\mu, \nu)}(z) = \phi_{(\mu_1, \nu_1)}(z) + \phi_{(\mu_2, \nu_2)}(z) \label{cum45}. 
\end{gather}
Let $(\mu_1 {}_{\nu_1}\!\boxplus_{\nu_2} \mu_2, \nu_1 \boxplus \nu_2 )$ denote the c-free convolution of $(\mu_1, \nu_1)$ and $(\mu_2, \nu_2)$.

\subsection{Technical facts}
We summarize the notation and several lemmata which will be used in Sections \ref{deform11}, \ref{Sec4} and \ref{Sec5}.     
Let $\pa$, $\pb$, $\pmo$, $\pc$, $\p+$ and $\ps$ be the set of probability measures, the set of probability measures with finite variance, the set of probability measures with finite moments of all orders,  the set of probability measures with compact supports, the set of probability measures on $[0, \infty)$ and the set of symmetric probability measures, respectively. 
The following lemma was proved in \cite{Maa}.
\begin{lemma}\label{lem2211}
A probability measure  $\mu $ belongs to $\pb$ if and only if $H_\mu $ has the 
representation
\begin{equation} \label{maa}
H_\mu (z) = a + z + \int_ {\real} \frac{1}{x-z}d\rho (x), 
\end{equation}
where $a \in \real$, $\rho$ a positive finite measure. $a$ and $\rho$ are determined uniquely. 
Furthermore, we have $\rho (\real) = \sigma ^2(\mu)$ and $a = -m(\mu)$, 
where $m(\mu)$ is the mean of $\mu$ and $\sigma ^2 (\mu)$ is the variance of $\mu$. 
\end{lemma}


We define $a(\mu):= \inf \{x \in \supp \mu \}$ and $b(\mu):= \sup \{x \in \supp \mu \}$. 
We note that  $-\infty \leq a(\mu) < \infty$ and $-\infty < b(\mu) \leq \infty$. The following lemmata \ref{estimate}-\ref{symm} were proved in \cite{Has2}.
\begin{lemma}\label{estimate}
Let $\nu$ and $\mu$ be probability measures. Then the following inequalities hold: \\
$(1)$ If $ \supp \nu \cap (-\infty, 0] \neq \emptyset $ and 
$ \supp \nu \cap [0, \infty) \neq \emptyset$, then  $a(\mu) \geq a(\nu \rhd \mu)$ and $b(\mu) \leq b(\nu \rhd \mu)$.  \\
$(2)$ If $ \supp \nu \subset (-\infty, 0]$, then $a(\mu) \geq a(\nu \rhd \mu)$ and  $b(\nu) + b(\mu) \leq b(\nu \rhd \mu)$.  \\
$(3)$ If $\supp \nu \subset [0, \infty)$, then $a(\nu) + a(\mu) \geq a(\nu \rhd \mu)$ and $b(\mu) \leq b(\nu \rhd \mu) $.  
\end{lemma}

\begin{lemma}\label{positive1} 
We use the notation in (\ref{nev}). For $\mu \in \pa$, the condition $\mu \in \p+$ is equivalent to $\supp \eta \subset [0, \infty)$ and $H_{\mu}(-0) \leq 0$. 
Moreover, if $\supp \eta \subset [0, \infty)$, the condition $H_{\mu}(-0) \leq 0$ is equivalent to 
\[\eta(\{0 \}) = 0, ~~~\int_0 ^{\infty}\frac{1}{x}d\eta (x) < \infty,~~~b + \int_0 ^{\infty}\frac{1}{x}d\eta (x) \leq 0. 
\]
\end{lemma}

\begin{lemma}\label{subordinator}
Let $\{\mu_t \}_{t \geq 0}$ be a weakly continuous $\rhd$-convolution semigroup 
with $\mu_0 = \delta_0$. Then the following statements are equivalent: 
\begin{itemize}
\item[$(1)$] there exists $t_0 > 0$ such that $\supp\mu_{t_0} \subset [0, \infty)$;
\item[$(2)$] $\supp \mu_t \subset [0, \infty)$ for all $0 \leq t < \infty$;
\end{itemize}
\end{lemma}

\begin{lemma}\label{symm}
We assume that the support of each $\mu_t$ is compact (or equivalently, the support of $\mu_t$ is compact for some $t>0$). 
Then the following statements are equivalent. 
\begin{itemize}
\item[$(1)$] There exists $t_0 > 0$ such that $\mu_{t_0}$ is symmetric.
\item[$(2)$] $\mu_{t}$ is symmetric for all $t > 0$.
\end{itemize}
\end{lemma}

\begin{lemma}\label{closed1}
$\p+$ and $\ps$ are closed subsets of $\pa$ under the weak topology.  
\end{lemma}
\begin{proof}
Let $\{\mu_n \} \subset \p+$ be a sequence converging to $\mu \in \pa$. 
The weak convergence implies that $\mu((-\infty, 0)) \leq \liminf \mu_n ((-\infty, 0)) = 0$; therefore, $\mu \in \p+$. 

For a probability measure $\nu$, $\nu \in \ps$ is equivalent to the condition 
\begin{equation}
\int_{\real} g(x)d\nu(x) = 0 \text{~for all~} g \in C_b(\real),~ g(-x) = -g(x), \end{equation}
where $C_b(\real)$ is the set of bounded continuous functions on $\real$. 
This equivalence can be proved with a simple approximation argument. 
Then the conclusion is not difficult. 
\end{proof}

\section{Conditionally monotone independence} \label{Condi}
It is known that the free, Boolean and monotone products of states (denoted by $\ast, \diamond$ and $\rhd$, respectively) can be expressed in terms of the c-free product \cite{BLS,Fra}, as follows. We consider triples of algebras and states $(\mathcal{A}_1, \varphi_1, \psi_1)$ and $(\mathcal{A}_2, \varphi_2, \psi_2)$. 
We assume that $\mathcal{A}_i$ has a decomposition 
\begin{equation}\label{assumption11}
\mathcal{A}_i = \comp 1 \oplus \mathcal{A}_i ^0
\end{equation} 
with a subalgebra $\mathcal{A}_i ^0$ $(i = 1, 2)$. Then the delta state $\delta_i$ $(i = 1, 2)$ is defined by $\delta_i(\lambda 1 + a^0) = \lambda$ for $\lambda \in \comp$ and $a^0 \in \mathcal{A}_i ^0$. $\delta_i$ is a homomorphism from $\mathcal{A}_i$ to $\comp$. Conversely, if there exists a homomorphism from $\mathcal{A}_i$ to $\comp$, then $\mathcal{A}_i ^0$ can be defined to be its kernel. 

We have the following relations.  
\begin{gather}
(\varphi, \varphi) \ast (\psi, \psi) = (\varphi \ast \psi, \varphi \ast \psi) \text{~on~} \mathcal{A}_1  \ast \mathcal{A}_2, \label{free2} \\
(\varphi, \delta_1) \ast (\psi, \delta_2) = (\varphi \diamond \psi, \delta_1 \ast \delta_2) \text{~on~} \mathcal{A}_1 ^0  \ast \mathcal{A}_2 ^0,  \label{boolean2}\\
(\varphi, \delta_1) \ast (\psi, \psi) = (\varphi \rhd \psi, \psi) \text{~on~} \mathcal{A}_1 ^0 \ast \mathcal{A}_2 \label{monotone2}. 
\end{gather}
In terms of additive convolutions of (compactly supported) probability measures, the equalities (\ref{free2})-(\ref{monotone2}) can be written as 
\begin{gather}
(\mu, \mu) \boxplus (\nu, \nu) = (\mu \boxplus \nu, \mu \boxplus \nu), \label{free1}\\
(\mu, \delta_0) \boxplus (\nu, \delta_0) = (\mu \uplus \nu, \delta_0), \label{boolean1}\\
(\mu, \delta_0) \boxplus (\nu, \nu) = (\mu \rhd \nu, \nu). \label{monotone1}
\end{gather}
We can understand the associative laws of the Boolean and free convolutions from (\ref{free1}) and (\ref{boolean1}) since the c-free convolution is associative. The associative law of the monotone convolution is, however, not easy to understand 
from (\ref{monotone1}) since we cannot repeat  (\ref{monotone1}) more than twice. We note here that the associative law of the monotone convolution was proved rigorously first in \cite{Fra4}. 

We show that the monotone convolution for the second component solves this problem.   We follow the setting of \cite{BFGKT} on unitization.
\begin{definition}\label{c-monotoneprod1}
(1) Let $(\mathcal{A}_1, \varphi_1, \psi_1)$ and $(\mathcal{A}_2, \varphi_2, \psi_2)$ be triples consisting of algebras 
and two linear functionals and let $\mathcal{A}_1 \ast_{nu} \mathcal{A}_2$ be the free product without the identification of units. We define a c-monotone product 
\begin{equation}
(\mathcal{A}_1, \varphi_1, \psi_1) \rhd (\mathcal{A}_2, \varphi_2, \psi_2) := (\mathcal{A}_1 \ast_{nu} \mathcal{A}_2, \varphi_1 \rhd_{\psi_2} \varphi_2, \psi_1 \rhd \psi_2 ), 
\end{equation}
by setting $\varphi_1 \rhd_{\psi_2} \varphi_2$ as follows. 
Denote the unitization of each $\mathcal{A}_i$ by $\widetilde{\mathcal{A}_i}:=\comp 1_{\widetilde{\mathcal{A}_i}} \oplus \mathcal{A}_i$. Then the linear functionals $\varphi_i$ and $\psi_i$ are naturally extended 
to $\widetilde{\varphi_i}$ and $\widetilde{\psi_i}$ on $\widetilde{\mathcal{A}_i}$ by  $\widetilde{\varphi_i}(1_{\widetilde{\mathcal{A}_i}}) = 1$ and $\widetilde{\psi_i}(1_{\widetilde{\mathcal{A}_i}}) = 1$.  
There is a natural isomorphism $\widetilde{\mathcal{A}_1 \ast_{nu} \mathcal{A}_2} \cong \widetilde{\mathcal{A}_1} \ast \widetilde{\mathcal{A}_2}$. 
We let $\widetilde{\delta_1}$ be the delta state associated to the decomposition $\widetilde{\mathcal{A}_1} = \comp 1 \oplus \mathcal{A}_1$.  We define $\varphi_1 \rhd_{\psi_2} \varphi_2$ to be the restriction of $\widetilde{\varphi_1}  {}_{\widetilde{\delta_1}}\!\!\ast_{\widetilde{\psi_2}} \widetilde{\varphi_2}$ on $\mathcal{A}_1 \ast_{nu} \mathcal{A}_2$. \\
(2) For pairs of probability measures $(\mu_1, \nu_1)$ and $(\mu_2, \nu_2)$, 
we define an additive c-monotone convolution $(\mu_1, \nu_1) \rhd (\mu_2, \nu_2) := (\mu_1 {}_{\delta_0}\!\!\boxplus_{\nu_2} \mu_2, \nu_1 \rhd \nu_2)$. 
We write $\mu_1 \rhd_{\nu_2} \mu_2$ for $\mu_1 {}_{\delta_0}\!\boxplus_{\nu_2}\mu_2$.  
\end{definition}
\begin{remark}  
We defined the c-monotone product on the free product without the identification of units. Considering the context of category theory \cite{BFGKT}, it is natural to treat products of linear functionals. This is why we considered linear functionals instead of states. The c-monotone product, however, preserves the positivity of linear functionals since it is defined to be the restriction of the c-free product. 
\end{remark}
\begin{proposition}\label{asso111}
$\mu_1 \rhd_{\nu_2} \mu_2$ is characterized by 
\begin{equation}
H_{\mu_1 \rhd_{\nu_2} \mu_2}=H_{\mu_1}\circ H_{\nu_2} + H_{\mu_2} - H_{\nu_2}. 
\end{equation}
\end{proposition}
\begin{proof}
We immediately obtain the equality $H_{\mu_1} \circ H_{\nu_2}(z) = H_{\nu_2}(z) - \phi_{(\mu_1, \delta_0)} (H_{\nu_2}(z))$ from (\ref{re14}). 
Since a c-free convolution is characterized by the sum of $\phi_{(,)}$, we have 
$H_{\mu_1} \circ H_{\nu_2}(z) - H_{\nu_2} (z) + H_{\mu_2}(z) = z - \phi_{(\mu_1 \rhd_{\nu_2} \mu_2, \nu_2)} \circ H_{\nu_2}(z)$. 
\end{proof}
\begin{remark}
For any probability measures $\mu_i,~ \nu_i$ ($i =1,~2$), we can prove that 
\begin{itemize}
\item[(1)] $H_{\mu_1}\circ H_{\nu_2} + H_{\mu_2} - H_{\nu_2}$ is an analytic map from $\com+$ to $\com+$, 
\item[(2)] $\inf_{z \in \com+} \frac{\im(H_{\mu_1}\circ H_{\nu_2}(z) + H_{\mu_2}(z) - H_{\nu_2} (z))}{\im~ z}  = 1$. 
\end{itemize}
Therefore, the definition of the c-monotone convolution of compactly supported probability measures can be extended to arbitrary probability measures \cite{Maa}.   
\end{remark}

We can easily check with Proposition \ref{asso111} that the c-monotone convolution of probability measures is associative, i.e., 
$\big{(}(\mu_1, \nu_1) \rhd (\mu_2, \nu_2) \big{)} \rhd (\mu_3, \nu_3) = (\mu_1, \nu_1) \rhd \big{(}(\mu_2, \nu_2)  \rhd (\mu_3, \nu_3)\big{)}$. Therefore, the c-monotone product of pairs of linear functionals is also expected to be associative. To prove this, we need to know how to compute mixed moments under the c-monotone product of linear functionals. We use the following notation: for a linearly ordered index set $I = \{i_1, i_2, \cdots, i_n \}$ with $i_1 < i_2 < \cdots < i_n$, we set 
\[\overrightarrow{\prod_{i \in I}} a_i := a_{i_1}a_{i_2} \cdots a_{i_n}. \]  

\begin{lemma}\label{formula11} Let $x_j$ and $y_k$ be (possibly non-commutative) elements in an algebra over $\comp$ with unit 1 and let $p_j \in \comp$. Then the following identity holds: 
\begin{equation}\label{eq98}
x_1 y_1 x_2 \cdots y_{n-1} x_n = \sum_{S \subset \{1, \cdots, n-1 \}} \Big{(} \prod_{j \notin S} p_j \Big{)} 
\Big{(} x_{S_1}(y_{k_1} -  p_{k_1}1)x_{S_2}(y_{k_2} -  p_{k_2}1) \cdots (y_{k_m} -  p_{k_m}1)x_{S_{m+1}} \Big{)},     
\end{equation}
 where $x_{S_j}:= \overrightarrow{\prod}_{k \in S_j} x_k$. $\{ S_j \}$ is a partition of $\{1, \cdots, n \}$ determined by $S$ as follows: if $S = \{k_1, \cdots, k_m \}$ with $1 \leq k_1 < \cdots < k_m \leq n-1$, then $S_j := \{k_{j-1} + 1, \cdots, k_j \}$ for $1 \leq j \leq m+1$, where $k_0 := 0$ and $k_{m+1}:= n$.  
If $S = \emptyset$, then $S_1 = \{1, \cdots, n \}$. A product over the empty set is defined to be $1$. 
\end{lemma}
\begin{proof}
(\ref{eq98}) can be proved easily by induction on the number $n$. 
\end{proof}
\begin{theorem} \label{rules1}
Let $(\mathcal{A}_1, \varphi_1, \psi_1)$ and $(\mathcal{A}_2, \varphi_2, \psi_2)$ be triples consisting of algebras 
and two linear functionals. 
The calculation rule for mixed moments under $\varphi_1 \rhd_{\psi_2} \varphi_2$ is what follows.  
\begin{itemize}
\item[$(1)$] $\varphi_1 \rhd_{\psi_2} \varphi_2 (b a x) = \varphi_2 (b) \varphi_1  \rhd_{\psi_2} \varphi_2(a x)  \text{~for~} a \in \mathcal{A}_1$, $b \in \mathcal{A}_2$ and $b a x \in \mathcal{A}_1 \ast_{nu}\mathcal{A}_2$. This is also true if $x$ is absent. 
\item[$(2)$] $\varphi_1 \rhd_{\psi_2} \varphi_2 (x a b) =  \varphi_1  \rhd_{\psi_2} \varphi_2(xa)\varphi_2 (b)  \text{~for~} a \in \mathcal{A}_1$, $b \in \mathcal{A}_2$ and $xab \in \mathcal{A}_1 \ast_{nu}\mathcal{A}_2$. This is also true if $x$ is absent.     
\item[$(3)$] $\varphi_1 \rhd_{\psi_2} \varphi_2 (a_1 b_1  \cdots  b_{n-1} a_n)  = (\varphi_2 (b_j) - \psi_2 (b_j)) \varphi_1 \rhd_{\psi_2} \varphi_2 (a_1 b_1 a_2 \cdots b_{j-1} a_j)  \varphi_1 \rhd_{\psi_2} \varphi_2 (a_{j+1} b_{j+1} \cdots  b_{n-1} a_n) \\ + \psi_2 (b_j) \varphi_1 \rhd_{\psi_2} \varphi_2 (a_1 b_1 a_2 \cdots b_{j-1} a_j a_{j+1} b_{j+1} \cdots  b_{n-1} a_n)$  for $a_k \in \mathcal{A}_1$, $b_k \in \mathcal{A}_2$, $1 \leq j \leq n-1$, $n \geq 2$.   
\end{itemize} 
\end{theorem} 
\begin{proof}
Since (1) and (2) can be proved similarly, we only prove (3). Moreover, we assume that $j = 1$ since the proof for general $j$ is essentially the same.  
We denote by $1$ the unit $1_{\widetilde{\mathcal{A}_1 \ast_{nu} \mathcal{A}_2}}$ for simplicity. First we obtain 
\begin{equation}
\begin{split}
\widetilde{\varphi_1} {}_{\widetilde{\delta}_1}\!\!\ast_{\widetilde{\psi_2}} \widetilde{\varphi_{2}} (a_1 b_1 a_2 \cdots b_{n-1} a_n) &= \widetilde{\varphi_1} {}_{\widetilde{\delta}_1}\!\!\ast_{\widetilde{\psi_2}} \widetilde{\varphi_{2}}(a_1 (b_1 - \psi_2 (b_1)1) a_2 \cdots b_{n-1} a_n) \\&~~~~~~~~~~~~ + \psi_2 (b_1)\widetilde{\varphi_1} {}_{\widetilde{\delta}_1}\!\!\ast_{\widetilde{\psi_2}} \widetilde{\varphi_{2}}(a_1 a_2 \cdots b_{n-1} a_n). 
\end{split}
\end{equation} 
Using the identity (\ref{eq98}) and the definition of c-free products, we have 
\begin{equation*}
\begin{split}
&\widetilde{\varphi_1} {}_{\widetilde{\delta}_1}\!\!\ast_{\widetilde{\psi_2}} \widetilde{\varphi_{2}} (a_1 (b_1 - \psi_2(b_1)1) a_2 \cdots b_{n-1} a_n) \\
&= \widetilde{\varphi_1} {}_{\widetilde{\delta}_1}\!\!\ast_{\widetilde{\psi_2}} \widetilde{\varphi_{2}}\Big{(} a_1 (b_1 - \psi_2(b_1)1) \sum_{S \subset \{2, \cdots, n-1 \}} \Big{(} \prod_{j \notin S} \psi_2(b_j) \Big{)} 
\Big{(} a_{S_1}(b_{k_1} -  \psi_2 ( b_{k_1})1) \cdots a_{S_{m+1}} \Big{)} \Big{)} \\ 
&= (\varphi_2(b_1) - \psi_2(b_1))  \sum_{S \subset \{2, \cdots, n-1 \}} \Big{(}\prod_{j \notin S} \psi_2 (b_j) \Big{)} \Big{(} \prod_{j \in S} (\varphi_2(b_j) - \psi_2 (b_j)) \Big{)} \varphi_1 (a_1)\prod_{j = 1}^{|S|+1} \varphi_1 (a_{S_j})  \\ 
&= \varphi_1 (a_1) (\varphi_2(b_1) - \psi_2(b_1)) \widetilde{\varphi_1} {}_{\widetilde{\delta}_1}\!\!\ast_{\widetilde{\psi_2}} \widetilde{\varphi_{2}} (a_2 b_2 \cdots b_{n-1} a_n). 
\end{split}
\end{equation*}  
\end{proof}
\begin{theorem}\label{assoc71}
The c-monotone product is associative. 
\end{theorem}
\begin{proof}
Let $(\mathcal{A}_i, \varphi_i, \psi_i)$ $(1 \leq i \leq 3)$ be triples consisting of algebras and two linear functionals. 
We can naturally identify $(\mathcal{A}_1 \ast_{nu} \mathcal{A}_2) \ast_{nu} \mathcal{A}_3$ with $\mathcal{A}_1 \ast_{nu} (\mathcal{A}_2 \ast_{nu} \mathcal{A}_3)$ 
and denote them simply by $\mathcal{A}_1 \ast_{nu} \mathcal{A}_2 \ast_{nu} \mathcal{A}_3$. 
What should be proved is the equality $\varphi_1 \rhd_{\psi_2 \rhd \psi_3} (\varphi_2 \rhd_{\psi_3} \varphi_3) = (\varphi_1 \rhd_{\psi_2} \varphi_2) \rhd_{\psi_3} \varphi_3$ on $\mathcal{A}_1 \ast_{nu} \mathcal{A}_2 \ast_{nu} \mathcal{A}_3$.  
We call $x \in \mathcal{A}_1 \ast_{nu} \mathcal{A}_2 \ast_{nu} \mathcal{A}_3$ an elementary word 
if $x$ is of such a form as $x = x_1 x_2 \cdots x_n$, $x_j \in \mathcal{A}_{i_j}$, $i_1 \neq \cdots \neq i_n$. 
Each $s_i$ is simply called an element of $x$. We give a proof by induction on the number of elements of $\mathcal{A}_3$ contained in an elementary word.  
In this proof we use the notation $a$, $a'$ and $a_j$ for elements of $\mathcal{A}_1$, $b$, $b'$ and $b_j$ for elements of $\mathcal{A}_2$ and $c$ and $c_j$ for elements of $\mathcal{A}_3$. 

For $x \in \mathcal{A}_2 \ast_{nu} \mathcal{A}_3 \subset \mathcal{A}_1 \ast_{nu} \mathcal{A}_2 \ast_{nu} \mathcal{A}_3$, we can easily prove that $\varphi_1 \rhd_{\psi_2 \rhd \psi_3} (\varphi_2 \rhd_{\psi_3} \varphi_3)(x) = (\varphi_1 \rhd_{\psi_2} \varphi_2) \rhd_{\psi_3} \varphi_3 (x)$. We assume that the statement is the case for 
all elementary words containing not more than $n$ elements of $\mathcal{A}_3$. Let $x$ be an elementary word 
containing $n + 1$ elements of $\mathcal{A}_3$. Let $c$ be the ($n+1$)-th element of $\mathcal{A}_3$ contained in $x$ and let $a$, if exists, be the last element of $\mathcal{A}_1$ which appears before $c$. 
Except for some cases, $x$ can be written as $x = uavcw$, where each component satisfies the following property: \\
(i) $u$ is an elementary word in $\mathcal{A}_1 \ast_{nu} \mathcal{A}_2 \ast_{nu} \mathcal{A}_3$ containing not more than $n$ elements of $\mathcal{A}_3$; \\
(i\hspace{-.1em}i) $v$ is an elementary word in $\mathcal{A}_2 \ast_{nu} \mathcal{A}_3$ containing not more than $n$ elements of $\mathcal{A}_3$. The last element of $v$ belongs to $\mathcal{A}_2$, i.e., $v = v' b_1$; \\
(i\hspace{-.1em}i\hspace{-.1em}i) $w$ is an elementary word in $\mathcal{A}_1 \ast_{nu} \mathcal{A}_2$. \\ 
Here we note some remarks on the exceptional cases. 
For $x$ of such a form as $x = uacw$, the following proof is applicable by understanding that $v$ plays the role of a unit. For $x$ of such a form as $x = uavc$, the proof follows easily from the property (1) in Theorem \ref{rules1}. If $a$ does not appear, again from the property (1) in Theorem \ref{rules1} the proof is easy. Therefore, we only prove the claim for $x$ of the form $uavcw$. We consider the following two cases: 
\begin{itemize}
\item[(1)] The first element in $w$ belongs to $\mathcal{A}_1$, i.e., $w = a'w'$;  
\item[(2)] The first element in $w$ belongs to $\mathcal{A}_2$, i.e., $w = b_2w'$. 
\end{itemize}

Case (1): We obtain 
\begin{equation}
\begin{split}
&(\varphi_1 \rhd_{\psi_2} \varphi_2) \rhd_{\psi_3} \varphi_3 (uavcw) \\
&~~~ = (\varphi_3(c) - \psi_3(c))(\varphi_1 \rhd_{\psi_2} \varphi_2) \rhd_{\psi_3} \varphi_3 (uav)(\varphi_1 \rhd_{\psi_2} \varphi_2) \rhd_{\psi_3} \varphi_3 (w) \\ 
&~~~~~~~~~ + \psi_3(c) (\varphi_1 \rhd_{\psi_2} \varphi_2) \rhd_{\psi_3} \varphi_3 (uavw) \\ 
&~~~ = (\varphi_3(c) - \psi_3(c))\varphi_1 \rhd_{\psi_2 \rhd \psi_3} (\varphi_2 \rhd_{\psi_3} \varphi_3) (uav)
\varphi_1 \rhd_{\psi_2} \varphi_2 (w) \\ 
&~~~~~~~~~ + \psi_3(c) \varphi_1 \rhd_{\psi_2 \rhd \psi_3} (\varphi_2 \rhd_{\psi_3} \varphi_3) (uavw) \\ 
&~~~ = (\varphi_3(c) - \psi_3(c))\varphi_1 \rhd_{\psi_2 \rhd \psi_3} (\varphi_2 \rhd_{\psi_3} \varphi_3) (uav)
\varphi_1 \rhd_{\psi_2} \varphi_2 (w) \\ 
&~~~~~~~~~ + \psi_3 (c)(\varphi_2 \rhd_{\psi_3} \varphi_3 (v) - \psi_2 \rhd \psi_3 (v))\varphi_1 \rhd_{\psi_2 \rhd \psi_3} (\varphi_2 \rhd_{\psi_3} \varphi_3) (ua) \varphi_1 \rhd_{\psi_2} \varphi_2(w) \\ 
&~~~~~~~~~ + \psi_3 (c)\psi_2 \rhd \psi_3 (v) \varphi_1 \rhd_{\psi_2 \rhd \psi_3} (\varphi_2 \rhd_{\psi_3} \varphi_3) (uaw), \\ 
\end{split}
\end{equation}
where we have used the assumption of induction and the property (2) in Theorem \ref{rules1}. 
On the other hand, we obtain 
\begin{equation}
\begin{split}
&\varphi_1 \rhd_{\psi_2 \rhd \psi_3} (\varphi_2 \rhd_{\psi_3} \varphi_3) (uavcw) \\
&~~~ = (\varphi_2 \rhd_{\psi_3} \varphi_3(vc) - \psi_2 \rhd \psi_3(vc)) \varphi_1 \rhd_{\psi_2 \rhd \psi_3} (\varphi_2 \rhd_{\psi_3} \varphi_3) (ua)\varphi_1 \rhd_{\psi_2 \rhd \psi_3} (\varphi_2 \rhd_{\psi_3} \varphi_3 )(w)  \\ 
&~~~~~~~~~+ \psi_2 \rhd \psi_3(vc)\varphi_1 \rhd_{\psi_2 \rhd \psi_3} (\varphi_2 \rhd_{\psi_3} \varphi_3) (uaw)  \\ 
&~~~ = \varphi_2 \rhd_{\psi_3} \varphi_3(v) \varphi_3(c)  \varphi_1 \rhd_{\psi_2 \rhd \psi_3} (\varphi_2 \rhd_{\psi_3} \varphi_3) (ua)\varphi_1 \rhd_{\psi_2 \rhd \psi_3} (\varphi_2 \rhd_{\psi_3} \varphi_3 )(w)  \\ 
&~~~~~~~~~ -  \psi_2 \rhd \psi_3(v) \psi_3(c)\varphi_1 \rhd_{\psi_2 \rhd \psi_3} (\varphi_2 \rhd_{\psi_3} \varphi_3) (ua)\varphi_1 \rhd_{\psi_2 \rhd \psi_3} (\varphi_2 \rhd_{\psi_3} \varphi_3 )(w)  \\ 
&~~~~~~~~~+ \psi_3(c) \psi_2 \rhd \psi_3(v) \varphi_1 \rhd_{\psi_2 \rhd \psi_3} (\varphi_2 \rhd_{\psi_3} \varphi_3) (uaw)  \\ 
&~~~ = (\varphi_3(c)-\psi_3(c)+\psi_3(c)) \varphi_1 \rhd_{\psi_2 \rhd \psi_3} (\varphi_2 \rhd_{\psi_3} \varphi_3) (uav)\varphi_1 \rhd_{\psi_2}\varphi_2 (w)  \\ 
&~~~~~~~~~ -  \psi_2 \rhd \psi_3(v) \psi_3(c)\varphi_1 \rhd_{\psi_2 \rhd \psi_3} (\varphi_2 \rhd_{\psi_3} \varphi_3) (ua)\varphi_1 \rhd_{\psi_2}\varphi_2 (w)  \\ 
&~~~~~~~~~+ \psi_3(c) \psi_2 \rhd \psi_3(v) \varphi_1 \rhd_{\psi_2 \rhd \psi_3} (\varphi_2 \rhd_{\psi_3} \varphi_3) (uaw)  \\ 
&~~~ = (\varphi_3(c)-\psi_3(c)) \varphi_1 \rhd_{\psi_2 \rhd \psi_3} (\varphi_2 \rhd_{\psi_3} \varphi_3) (uav)\varphi_1 \rhd_{\psi_2} \varphi_2 (w)  \\ 
&~~~~~~~~~ +  \psi_3(c) \varphi_1 \rhd_{\psi_2 \rhd \psi_3} (\varphi_2 \rhd_{\psi_3} \varphi_3) (ua) \Big{(}\varphi_2 \rhd_{\psi_3} \varphi_3 (v) - \psi_2 \rhd \psi_3(v)\Big{)}\varphi_1 \rhd_{\psi_2}\varphi_2 (w)  \\ 
&~~~~~~~~~+ \psi_3(c) \psi_2 \rhd \psi_3(v) \varphi_1 \rhd_{\psi_2 \rhd \psi_3} (\varphi_2 \rhd_{\psi_3} \varphi_3) (uaw).  \\ 
\end{split}
\end{equation}
Therefore, they coincide with each other. 

Case (2): Carrying out a similar calculation, we have 
\begin{equation}
\begin{split}
&(\varphi_1 \rhd_{\psi_2} \varphi_2) \rhd_{\psi_3} \varphi_3 (uavcw) \\
&~~~ = (\varphi_3(c) - \psi_3(c))\varphi_1 \rhd_{\psi_2 \rhd \psi_3} (\varphi_2 \rhd_{\psi_3} \varphi_3) (uav)
\varphi_1 \rhd_{\psi_2} \varphi_2 (w) \\ 
&~~~~~~~~~ + \psi_3(c) \varphi_1 \rhd_{\psi_2 \rhd \psi_3} (\varphi_2 \rhd_{\psi_3} \varphi_3) (uavw) \\ 
&~~~ = (\varphi_3(c) - \psi_3(c))\varphi_1 \rhd_{\psi_2 \rhd \psi_3} (\varphi_2 \rhd_{\psi_3} \varphi_3) (ua)\varphi_2 \rhd_{\psi_3} \varphi_3 (v)
\varphi_1 \rhd_{\psi_2} \varphi_2 (w) \\ 
&~~~~~~~~~ + \psi_3 (c)(\varphi_2 \rhd_{\psi_3} \varphi_3 (vb_2) - \psi_2 \rhd \psi_3 (vb_2))\varphi_1 \rhd_{\psi_2 \rhd \psi_3} (\varphi_2 \rhd_{\psi_3} \varphi_3) (ua) \varphi_1 \rhd_{\psi_2} \varphi_2 (w') \\ 
&~~~~~~~~~ + \psi_3 (c)\psi_2 \rhd \psi_3 (vb_2) \varphi_1 \rhd_{\psi_2 \rhd \psi_3} (\varphi_2 \rhd_{\psi_3} \varphi_3) (uaw'). \\ 
\end{split}
\end{equation}
On the other hand we have  
\begin{equation}
\begin{split}
&\varphi_1 \rhd_{\psi_2 \rhd \psi_3} (\varphi_2 \rhd_{\psi_3} \varphi_3) (uavcw) \\
&~~~ = (\varphi_2 \rhd_{\psi_3} \varphi_3 (vcb_2) - \psi_2 \rhd \psi_3 (vcb_2)) \varphi_1 \rhd_{\psi_2 \rhd \psi_3} (\varphi_2 \rhd_{\psi_3} \varphi_3) (ua)\varphi_1 \rhd_{\psi_2} \varphi_2 (w')  \\ 
&~~~~~~~~~+ \psi_2 \rhd \psi_3 (vcb_2) \varphi_1 \rhd_{\psi_2 \rhd \psi_3} (\varphi_2 \rhd_{\psi_3} \varphi_3) (uaw')  \\ 
&~~~ = \varphi_2 \rhd_{\psi_3} \varphi_3 (v)(\varphi_3 (c) - \psi_3 (c))\varphi_2 (b_2)\varphi_1 \rhd_{\psi_2 \rhd \psi_3} (\varphi_2 \rhd_{\psi_3} \varphi_3) (ua)\varphi_1 \rhd_{\psi_2} \varphi_2 (w')  \\ 
&~~~~~~~~~ + \psi_3 (c)\varphi_2 \rhd_{\psi_3} \varphi_3 (vb_2) \varphi_1 \rhd_{\psi_2 \rhd \psi_3} (\varphi_2 \rhd_{\psi_3} \varphi_3) (ua)\varphi_1 \rhd_{\psi_2} \varphi_2 (w')  \\
&~~~~~~~~~- \psi_2 \rhd \psi_3 (vcb_2)\varphi_1 \rhd_{\psi_2 \rhd \psi_3} (\varphi_2 \rhd_{\psi_3} \varphi_3) (ua)\varphi_1 \rhd_{\psi_2}\varphi_2 (w')  \\ 
&~~~~~~~~~+ \psi_2 \rhd \psi_3 (vcb_2) \varphi_1 \rhd_{\psi_2 \rhd \psi_3} (\varphi_2 \rhd_{\psi_3} \varphi_3) (uaw')  \\ 
&~~~ = \varphi_2 \rhd_{\psi_3} \varphi_3 (v)(\varphi_3 (c) - \psi_3 (c)) \varphi_1 \rhd_{\psi_2 \rhd \psi_3} (\varphi_2\rhd_{\psi_3} \varphi_3) (ua)\varphi_1 \rhd_{\psi_2}\varphi_2 (w)  \\ 
&~~~~~~~~~ + \psi_3 (c)\varphi_2 \rhd_{\psi_3} \varphi_3 (vb_2)  \varphi_1 \rhd_{\psi_2 \rhd \psi_3} (\varphi_2 \rhd_{\psi_3} \varphi_3) (ua)\varphi_1 \rhd_{\psi_2}\varphi_2 (w')  \\ 
&~~~~~~~~~- \psi_3 (c)\psi_2 \rhd \psi_3 (vb_2)\varphi_1 \rhd_{\psi_2 \rhd \psi_3} (\varphi_2 \rhd_{\psi_3} \varphi_3) (ua)\varphi_1 \rhd_{\psi_2}\varphi_2 (w')  \\ 
&~~~~~~~~~+ \psi_3 (c)\psi_2 \rhd \psi_3 (vb_2) \varphi_1 \rhd_{\psi_2 \rhd \psi_3} (\varphi_2 \rhd_{\psi_3} \varphi_3) (uaw').  \\ 
\end{split}
\end{equation}
Then $\varphi_1 \rhd_{\psi_2 \rhd \psi_3} (\varphi_2 \rhd_{\psi_3} \varphi_3) (x) = (\varphi_1 \rhd_{\psi_2} \varphi_2) \rhd_{\psi_3} \varphi_3 (x)$ for any $x$ containing $n+1$ elements of $\mathcal{A}_3$. By induction, $\varphi_1 \rhd_{\psi_2 \rhd \psi_3} (\varphi_2 \rhd_{\psi_3} \varphi_3) (x) = (\varphi_1 \rhd_{\psi_2} \varphi_2) \rhd_{\psi_3} \varphi_3 (x)$ holds for all elementary words $x$ in $\mathcal{A}_1 \ast_{nu} \mathcal{A}_2 \ast_{nu} \mathcal{A}_3$. 
\end{proof}
Using the associativity of the c-monotone product, we can define the product $\rhd_{i \in I} (\mathcal{A}_i, \varphi_i, \psi_i)$ for a linearly ordered set $I$. 
Now we come to define the concept of c-monotone independence. 
\begin{definition}
Let $(\mathcal{A}, \varphi, \psi)$ be an algebraic probability space equipped with two states. We assume that $\mathcal{A}$ is unital. 
Let $I$ be a linearly ordered set. A family of subalgebras $\{\mathcal{A}_i \}_{i \in I}$ is said to be c-monotone independent if 
the following properties are satisfied: 
\begin{itemize}
\item[$(1)$] $\varphi (a_1 a_2 \cdots a_n ) = \varphi (a_1) \varphi(a_2 \cdots a_n) \text{~for~} a_k \in \mathcal{A}_{i_k}$ with $i_1, i_2, \cdots, i_n \in I$ and $i_1 > i_2$;    
\item[$(2)$] $\varphi (a_1 a_2 \cdots a_n ) = \varphi (a_n) \varphi(a_1\cdots a_{n-1}) \text{~for~} a_k \in \mathcal{A}_{i_k}$ with $i_1, i_2, \cdots, i_n \in I$ and $i_n > i_{n-1}$;   
\item[$(3)$]  $\varphi (a_1 a_2 \cdots a_n)  = (\varphi (a_j) - \psi (a_j)) \varphi (a_1 \cdots a_{j-1})  \varphi (a_{j+1} \cdots a_n) + \psi (a_j) \varphi (a_1a_2 \cdots a_{j-1} a_{j+1} \cdots a_n)$ \\ for $a_k \in \mathcal{A}_{i_k}$ with $i_1, i_2, \cdots, i_n \in I $ and $i_{j-1} < i_j > i_{j+1}$ with $2 \leq j \leq n-1$; 
\item[$(4)$] $\mathcal{A}_i$ are monotone independent with respect to $\psi$.  
\end{itemize} 
\end{definition}
\begin{remark}
(1) This definition includes monotone  (resp. Boolean) independence in the special case $\varphi = \psi$ (resp. $\psi = 0$) on $\mathcal{A}_i$ for all $i \in I$. \\ 
(2) One can define the independence for random variables by considering the subalgebra generated by each random variable without the unit. \\
(3) Let $a_1$ and $a_2$ be c-monotone independent and self-adjoint. 
If $a_i$ has a pair of probability distributions $(\mu_i,\nu_i)$ under a pair of states, then the distribution of $a_1+a_2$ is $(\mu_1,\nu_1)\rhd(\mu_2,\nu_2)$. This fact follows from the definition of the c-monotone convolution. 
\end{remark}
We can prove that $\mathcal{A}_i$ ($i \in I$) are c-monotone independent in $(\mathcal{A}, \varphi, \psi) = \rhd_{i \in I} (\mathcal{A}_i, \varphi_i, \psi_i)$.  

Here we mention a relation between c-monotone/free independence and orthogonal independence. 
The reader is referred to \cite{Len1} for the definition and properties of orthogonal independence and the orthogonal (additive) convolution. 
\begin{proposition}
With the notation in Definition \ref{c-monotoneprod1}, $\mathcal{A}_1$ is orthogonal to $\mathcal{A}_2$ with respect to  the linear functional $(\varphi_1 \rhd_{\psi_2} 0_2, \psi_1 \rhd \psi_2)$ in the algebra $\mathcal{A} := \mathcal{A}_1 \ast_{nu} \mathcal{A}_2$ for arbitrary linear functionals $\varphi_1, \psi_1, \psi_2$. 
$0_2$  denotes the $0$ linear functional on $\mathcal{A}_2$. 
\end{proposition}
\begin{remark}
Orthogonal independence has been defined in a unital algebra \cite{Len1}. To state this 
proposition strictly along the line of the original definition, we only need to unitize $\mathcal{A}$ and extend the linear functionals to $\widetilde{\mathcal{A}}$.
\end{remark} 
\begin{proof}
This fact follows from Theorem \ref{rules1} immediately. 
\end{proof}
In terms of convolutions of probability measures, we obtain the following relations. 
\begin{gather}
(\mu, \delta_0) \boxplus (\delta_0, \nu) = (\mu \vdash \nu, \nu), \label{ortho1} \\
(\mu, \lambda) \rhd (\delta_0, \nu) = (\mu \vdash \nu, \lambda \rhd \nu). \label{ortho2} 
\end{gather} 
The orthogonal convolution is characterized by $H_{\mu \vdash \nu} (z) = H_\mu(H_\nu(z)) - H_\nu(z) + z$ (see Theorem 6.2 of \cite{Len1}).   
If we use the relation in Proposition \ref{asso111}, this characterization follows immediately.  

We give another application: the equality $\mu \rhd \nu = (\mu \vdash \nu) \uplus \nu$ (see Corollary 6.6 in \cite{Len1}) can be understood in terms of the associativity of the c-monotone convolution. 
We first note that $(\delta_0, \nu) \rhd (\nu, \delta_0) = (\delta_0 {}_{\delta_0}\!\! \boxplus _{\delta_0} \nu, \nu) = (\nu, \nu)$ and $(\mu, \mu) \rhd (\delta_0, \nu) = (\mu \vdash \nu, \mu)$. 
Using the associativity we can calculate $(\mu, \mu) \rhd (\delta_0, \nu) \rhd (\nu, \delta_0)$ in different two ways and we obtain $((\mu \vdash \nu)\uplus \nu, \mu \rhd \nu) = (\mu \rhd \nu, \mu \rhd \nu)$.

\section{Conditionally monotone cumulants}\label{Cumulants112}
\subsection{Power additivity of cumulants}\label{condmonocum1} 
In the paper \cite{H-S}, monotone cumulants were introduced. The crucial concept is the dot operation $N.X$ which is defined as the sum of independent random variables with the same distributions as $X$. If $X$ is a self-adjoint element with the distribution $\mu$, the random variable $N.X$ has the probability 
distribution $\mu^{\star N}$, where $\star$ is the additive convolution associated to a concept of independence. Moreover, the uniqueness of cumulants holds also in the setting of two states (we explain this briefly at the end of this subsection). 

While the dot operation is fundamental for cumulants, we use a method based on the characterization in Proposition \ref{asso111} to define cumulants. 
This method is useful to obtain relations between generating functions. 

The c-monotone convolution consists of two components. 
While the first component comes from a c-free convolution, c-free cumulants are not useful to find c-monotone cumulants. 
For the second component, we use the monotone cumulants. 

We define c-monotone cumulants for the first component. Let $D_\lambda$ be the dilation operator defined so that $\int_{\real}f(x) (D_\lambda \mu) (dx) = \int_{\real}f(\lambda x)\mu(dx)$ for all bounded continuous functions $f$. 
\begin{lemma}\label{cum11}
 Let $\mu \in \pmo$. We define $b_n = b_n(\mu)$ by  
$H_{\mu}(z) = z \Big(1 + \sum_{n = 1}^{\infty} \frac{b_n}{z^{n}} \Big{)}$ in the sense of asymptotic expansion. 
Then $b_n$ is of the form 
\begin{equation}\label{cum14}
b_n = -m_n + W_n(m_{1}, \cdots, m_{n-1}),
\end{equation}
where $W_n$ is a polynomial. Moreover, $b_n(D_\lambda \mu) = \lambda^n b_n(\mu)$ for $\lambda > 0$. 
In particular, $W_n$ does not contain a constant term or linear terms. 
\end{lemma} 
\begin{proof}
We have 
\begin{equation}
\begin{split}
H_{\mu}(z) &= z\frac{1}{1 + \sum_{k = 1} ^{\infty}\frac{m_k}{z^k}} \\
           &= z\Big{(}1 - \sum_{k = 1} ^{\infty}\frac{m_k}{z^k} + (\sum_{k = 1} ^{\infty}\frac{m_k}{z^k})^2 - \cdots  \Big{)}, \\ 
\end{split}
\end{equation}
so that the polynomials $W_n$ exist. The last statement follows from the relation $H_{D_\lambda \mu}(z) = \lambda H_{\mu}(\lambda^{-1}z)$. 
\end{proof} 

\begin{lemma}
For any $n \geq 3$, there exists a polynomial $Y_n$ of $2n-3$ variables such that $b_n(\mu_1 \rhd_{\nu_2} \mu_2) = b_n(\mu_1) + b_n(\mu_2) + Y_n(b_1(\mu_1), \cdots, b_{n - 1}(\mu_1), b_1(\nu_2), \cdots, b_{n-2}(\nu_2))$. In addition, 
$Y_n$ does not contain a constant term or linear terms. 
For $n = 1$ and $2$, we have 
$b_1(\mu_1 \rhd_{\nu_2} \mu_2) = b_1(\mu_1) + b_1(\mu_2)$ and $b_2(\mu_1 \rhd_{\nu_2} \mu_2) = b_2(\mu_1) + b_2(\mu_2)$. 
\end{lemma}
\begin{proof}
We have the equality 
\begin{equation}
\begin{split}
H_{\mu_1 \rhd_{\nu_2} \mu_2}(z) &= H_{\mu_1} (H_{\nu_2}(z)) + H_{\mu_2}(z) - H_{\nu_2}(z) \\
                   &= b_1(\mu_1) + \sum_{n = 1} ^{\infty} b_{n+1}(\mu_1) G_{\nu_2}(z)^{n} + z + b_1(\mu_2) + \sum_{n = 1}^{\infty} \frac{b_{n+1}(\mu_2)}{z^n} \\ 
                   &= b_1(\mu_1) + \sum_{n = 1} ^{\infty} \frac{b_{n+1}(\mu_1)}{z^n} \Big{(}1 + \sum_{k=1}^{\infty} \frac{m_k(\nu_2)}{z^k} \Big{)}^{n} + z + b_1(\mu_2) + \sum_{n = 1}^{\infty} \frac{b_{n+1}(\mu_2)}{z^n}. \\  
\end{split}
\end{equation} 
We can express $m_k$ as a polynomial of $r_n$ ($1 \leq p \leq k$) from Lemma \ref{cum11}: $m_k(\nu_2) = -b_k(\nu_2) + X_k(b_1(\nu_2), \cdots, b_{k-1}(\nu_2))$, 
where $X_k$ is a polynomial without a constant term or linear terms. Therefore, we have the conclusion.  
\end{proof}

Let $r_n(\nu)$ be the monotone cumulants of $\mu$ and $A_\nu(z) = -\sum_{n=1}^{\infty}\frac{r_n(\nu)}{z^{n-1}}$ be their generating function.  
Let $\{(\mu_t, \nu_t) \}_{t \geq 0}$ be a c-monotone convolution semigroup with $(\mu_0, \nu_0) = (\delta_0, \delta_0)$. Then $\{\nu_t \}_{t \geq 0}$ is a monotone convolution semigroup with $\nu_0 = \delta_0$. We define $\mu:= \mu_1$ and $\nu = \nu_1$. 
 Proposition \ref{asso111} implies that $H_{\mu_{s+t}} = H_{\mu_s} \circ H_{\nu_t} + H_{\mu_t} - H_{\nu_t}$. 
We assume that $m_n(\mu_t)$ is a differentiable function of $t$. 
Differentiating the equality with $s$ formally and putting $s = 0$, we obtain 
\begin{equation}\label{cum13}
 \frac{d}{dt} H_{\mu_t}(z) = A_{(\mu, \nu)}(H_{\nu_t}(z)), 
\end{equation}  
where $A_{(\mu, \nu)}(z) = \frac{d}{dt}H_{\mu_t}(z)|_{t = 0}$. 
We expand $A_{(\mu, \nu)}(z)$ as a formal power series: 
\begin{equation}\label{cum15}
A_{(\mu, \nu)}(z) = -\sum_{n = 1}^{\infty} \frac{r _n (\mu, \nu)}{z^{n-1}}.
\end{equation} 
(\ref{cum13}) is equivalent to 
\begin{equation}\label{cum16}
\frac{d}{dt}G_{\mu_t}(z) = \sum_{n = 1} ^{\infty} r_n(\mu, \nu) G_{\mu_t}(z)^{2} G_{\nu_t}(z)^{n-1} 
\end{equation}  
as a formal power series. Thus there exists a polynomial $M_n$  
for any $n \geq 3$ such that 
\begin{equation}\label{cum17}
\begin{split}
\frac{d}{dt}m_n(\mu_t) &= r_n(\mu, \nu) + M_n(r_{1}(\mu, \nu), \cdots, r_{n-1}(\mu, \nu), m_1(\mu_t), \cdots, m_{n-1}(\mu_t), \\&~~~~~~~~~~~~~~~~~~~~~~~~~ m_1(\nu_t), \cdots, m_{n-2}(\nu_t)). 
\end{split}
\end{equation} 
For $n=1$ and $2$, we can understand that $M_1 = 0$ and that $M_2$ only depends on $r_1(\mu, \nu)$ and $m_1(\mu_t)$. 
Since $m_k(\nu_t)$ is a polynomial of $r_1(\nu),\cdots, r_k(\nu)$, we can prove inductively that $m_n(\mu_t)$ is of the form 
\begin{equation}
m_n(\mu_t) = r_n(\mu, \nu) t + t^2C_n(t, r_1(\mu, \nu), \cdots, r_{n-1}(\mu, \nu), r_1(\nu), \cdots, r_{n-2}(\nu)), 
\end{equation}
where $C_n$ is a polynomial for any $n \geq 3$ and $r_k(\mu)$ is the $k$-th monotone cumulant of $\mu$. From the construction, $C_n$ is a universal polynomial in the sense that $C_n$ does not depend on $\mu$ or $\nu$; $C_n$ is determined only by the definition of c-monotone independence.  
It is easy to show that $m_1(\mu_t) = r_1(\mu, \nu)t$ and $m_2(\mu_t)= r_2(\mu, \nu)t + r_1(\mu,\nu)^2t^2$. 
If we set $t=1$, $r_n(\mu, \nu)$ turns out to be expressed as a polynomial of $m_k(\mu)$ and $m_k(\nu)$ $(1 \leq k \leq n)$. 
This relation can be generalized to any probability measures $\mu,~\nu \in \pmo$ since $C_n$ is universal. 
\begin{definition}\label{momentcumulant}
For probability measures $\mu,~\nu \in \pmo$, we define the c-monotone cumulants $r_n(\mu, \nu)$ ($n \geq 1$) by the equations 
\[
m_n(\mu) = r_n(\mu,\nu) + C_n(1, r_1(\mu, \nu), \cdots, r_{n-1}(\mu, \nu), r_1(\nu), \cdots, r_{n-2}(\nu))
\] 
for $n \geq 3$. For $n = 1,2$, we define $r_1(\mu, \nu) := m_1(\mu)$ and $r_2(\mu, \nu) := m_2(\mu) - m_1(\mu)^2$.  
\end{definition}

Now we prove the power additivity of cumulants. The proof is based on the relation of generating functions in Proposition \ref{asso111}. We note, however, that we can give a proof without the use of generating functions 
as shown in \cite{H-S}.  
\begin{theorem}
$r_n((\mu, \nu)^{\rhd N}) = N r_n(\mu, \nu)$ holds for $\mu, ~\nu \in \pmo$ and $N \in \nat$. 
\end{theorem}
\begin{proof}
We define $\mu_N$ by $(\mu_N, \nu^{\rhd N}) = (\mu, \nu)^{\rhd N}$. 
We define 
\begin{equation}\label{cum22}
 m_n(\nu, t) := \sum_{k =1} ^{n} \sum_{1 = i_0 < i_1 < \cdots < i_{k-1} < i_k = n +1} 
\frac{t^k}{k!}\prod_{l = 1} ^{k} i_{l-1} r_{i_l - i_{l-1}}(\nu) 
\end{equation}
and 
\begin{equation}\label{cum18}
m_n(\mu,\nu, t) := r_n(\mu, \nu)t + t^2C_n(t, r_1(\mu, \nu), \cdots, r_{n-1}(\mu, \nu), r_1(\nu), \cdots, r_{n-2}(\nu)), 
\end{equation}
which may not be moments of a probability measure for general $t$. It is noted that $m_n(\mu, \nu, 1) = m_n(\mu)$ and $m_n(\mu, \mu, t) = m_n(\mu, t)$. 
The latter equality comes from the relation $(\mu, \mu) \rhd (\nu, \nu) = (\mu \rhd \nu, \mu \rhd \nu)$. We shall show that $m_n(\mu,\nu, N) = m_n(\mu_N, \nu^{\rhd N}, 1) (= m_n(\mu_N))$ for any $n$, $N \geq 1$. 
We define formal power series $A_{(\mu, \nu)}(z)$ using (\ref{cum15}), $H_{(\mu, \nu)}(t, z):= \Big{(}\sum_{n=0} ^{\infty}\frac{m_n (\mu, \nu, t)}{z^{n+1}} \Big{)}^{-1}$ 
and $H_\nu(t, z) :=  \Big{(}\sum_{n=0} ^{\infty}\frac{m_n (\nu, t)}{z^{n+1}} \Big{)}^{-1}$. 
Now we prove the equalities 
\begin{align}
&H_{(\mu, \nu)}(t+s,z) = H_{(\mu, \nu)}(t, H_\nu(s, z)) - H_\nu(s, z) + H_{(\mu, \nu)}(s, z) \label{eq50},  \\
&H_\nu(t+s,z) = H_\nu(t, H_\nu(s, z)),  
\end{align}
as formal power series. To do so, we define $K_s^{1}(t,z) := H_{(\mu, \nu)}(t, H_\nu(s, z)) - H_\nu(s, z) + H_{(\mu, \nu)}(s, z)$ and $K_s^2(t,z):=H_\nu(t,H_\nu(s,z))$.  
With the arguments just before Definition \ref{momentcumulant}, we have
\begin{equation}\label{cum23}
\begin{split}
&\frac{d}{dt} H_{(\mu, \nu)}(t,z) = A_{(\mu, \nu)}(H_{\nu}(t, z)), \\
&\frac{d}{dt} H_{\nu}(t,z) = A_{\nu}(H_{\nu}(t, z)). 
\end{split}
\end{equation} 
The latter equality follows from the former by setting $\mu = \nu$. Therefore it is clear that $(H_{(\mu, \nu)}(t+s, z), H_\nu(t+s, z))$ satisfies the two-dimensional complex differential equation 
 \begin{equation}
\begin{split}
&\frac{d}{dt} H_{(\mu, \nu)}(t+s,z) = A_{(\mu, \nu)}(H_{\nu}(t+s, z)), \\
&\frac{d}{dt} H_{\nu}(t+s,z) = A_{\nu}(H_{\nu}(t+s, z)) 
\end{split}
\end{equation} 
for a fixed $s$ with the initial value $(H_{(\mu, \nu)}(s, z), H_\nu(s, z))$. 
On the other hand,  (\ref{cum23}) implies that 
 \begin{equation}
\begin{split}
&\frac{d}{dt} K_s^{1}(t, z) = A_{(\mu,\nu)}(K^2_s(t,z)), \\
&\frac{d}{dt} K_{s}^2 (t,z) = A_{\nu}(K_{s}^2 (t, z)) 
\end{split}
\end{equation} 
with the initial value $(K_s^1(0, z), K_s^2(0, z)) = (H_{(\mu, \nu)}(s,z), H_\nu(s,z))$. Thus both $(H_{(\mu, \nu)}(t+s, z), H_\nu(t+s, z))$ and ($K_s^1(t,z), K_s^2(t,z)$) satisfy the same differential equation with the same initial value. 
It is not difficult to prove the uniqueness of solution of an ordinary differential equation as a formal power series; therefore $(H_{(\mu, \nu)}(t+s, z), H_\nu(t+s, z))=(K_s^1(t,z), K_s^2(t,z))$. 
Setting $t = s = 1$, we obtain $m_n(\mu, \nu, 2) = m_n(\mu_2, \nu \rhd \nu, 1) = m_n(\mu_2)$. Inductively, we can show that $m_n(\mu, \nu, N) = m_n(\mu_N)$. By using the (power) associativity of the c-monotone convolution, 
we also obtain $m_n(\mu, \nu, MN) = m_n(\mu_N, \nu^{\rhd N}, M)$ for $M, N \in \nat$. 
Then we have 
\begin{equation}
\begin{split}
&MNr_n(\mu, \nu) + M^2N^2C_n(MN, r_1(\mu, \nu), \cdots, r_{n-1}(\mu, \nu), r_1(\nu), \cdots, r_{n-2}(\nu)) \\ &~= Mr_n(\mu_N, \nu^{\rhd N}) + M^2C_n(M, r_1(\mu_N, \nu^{\rhd N}), \cdots, r_{n-1}(\mu_N, \nu^{\rhd N}), r_1(\nu^{\rhd N}), \cdots, r_{n-2}(\nu^{\rhd N})). 
\end{split}
\end{equation}
Since the coefficients of $M$ coincide, it holds that $r_n(\mu_N,\nu^{\rhd N}) = Nr_n(\mu, \nu)$. 
\end{proof}

Before closing this subsection, we note some remarks. 
First we summarize the basic relations between generating functions:  
\begin{gather}
\frac{d}{dt}H_{(\mu,\nu)}(t,z) = A_{(\mu, \nu)}(H_\nu(t, z)), ~H_{(\mu, \nu)}(1, z) = H_\mu(z), ~H_{(\mu, \nu)}(0, z) = z, \\
\frac{d}{dt}H_\nu(t, z) = A_\nu(H_\nu(t,z)), ~H_\nu(1, z) = H_\nu(z),~ H_\nu(0, z) = z.   
\end{gather}
These relations are analogous to 
\begin{gather}
H_{\mu}(z) = z - \phi_{(\mu, \nu)}(H_{\nu}(z)), \\ 
H_{\nu}(z) = z - \phi_{\nu}(H_{\nu}(z)),   
\end{gather}
which are basic  for c-free convolutions. 
We note that monotone cumulants and Boolean cumulants are special cases of c-monotone cumulants: $A_{(\mu, \mu)}$ is a generating function of monotone cumulants and $A_{(\mu, \delta_0)}$ is a generating function of 
Boolean cumulants. 

Second, we note that c-monotone cumulants $r_n(\mu,\nu)$ satisfy the following conditions. 
\begin{itemize}
\item[(K1')] Power additivity: $r_n((\mu,\nu)^{\rhd N})= Nr_n(\mu,\nu)$.
\item[(K2)] Homogeneity: for any $\lambda > 0$ and any $n$,  
\begin{equation}
r_n(D_{\lambda}\mu, D_\lambda \nu ) = \lambda ^n r_n(\mu, \nu), 
\end{equation}
where $D_\lambda$ is defined by $(D_\lambda \mu)(B) = \mu(\lambda^{-1}B)$ for any Borel set $B$.  
\item[(K3)] For any $n$, there exists a universal polynomial $Q_n$ (in the sense that $Q_n$ does not depend on $\mu$ or $\nu$) of $2n-2$ variables such that 
\begin{equation}
r_n(\mu, \nu) = m_n(\mu) + Q_n(m_p(\mu),m_q(\nu) ~(1 \leq p,q \leq n-1)). 
\end{equation}
\end{itemize}
The usual additivity $r_n((\mu_1,\nu_1)\rhd(\mu_2,\nu_2)) = r_n(\mu_1,\nu_1)+ r_n(\mu_2,\nu_2)$ does not hold due to the non-commutativity of the convolution. 
We can easily prove the uniqueness of cumulants satisfying (K1'), (K2) and (K3) (see \cite{H-S} and also Section \ref{Cum} of the present paper).

\subsection{Moment-cumulant formula and monotone partitions}
Next we show a combinatorial moment-cumulant formula for the c-monotone convolution. 
Let $\lncpn$ be the set of all linearly ordered non-crossing partitions defined by 
\begin{equation}
\lncpn := \{(\pi, \lambda): \pi \in \ncpn, \lambda \text{~is a linear ordering of the blocks of $\pi$} \}.   
\end{equation}

We prove that the $n$-th moment can be described as  
\begin{equation}\label{eq200}
m_n(\mu) = \sum_{(\pi, \lambda) \in \mpn} \frac{1}{|\pi|!}\Big{(} \prod_{\substack{V \in \pi, \\ V: \text{\normalfont ~outer}}}r_{|V|}(\mu, \nu) \Big{)} \Big{(}  \prod_{\substack{V \in \pi, \\ V: \text{\normalfont ~inner}}}r_{|V|}(\nu) \Big{)}.
\end{equation}
This formula is analogous to the c-free formula    
\begin{equation}\label{eq201}
m_n(\mu) = \sum_{(\pi, \lambda) \in \lncpn} \frac{1}{|\pi|!}\Big{(} \prod_{\substack{V \in \pi, \\ V: \text{~outer}}}R_{|V|}(\mu, \nu) \Big{)} \Big{(}  \prod_{\substack{V \in \pi, \\ V: \text{~inner}}}R_{|V|}(\nu) \Big{)}
\end{equation}
if we impose the linear order structure. Clearly the role of the linear order structure is trivial in (\ref{eq201}), but crucial in (\ref{eq200}). 

When we defined c-monotone cumulants, we used the Taylor expansion of the equality $\frac{d}{dt} H_{\mu_t}(z) = A_{(\mu, \nu)}(H_{\nu_t}(z))$  to prove 
the power additivity $r_n((\mu, \nu)^{\rhd N}) = N r_n(\mu, \nu)$. 
By contrast, the following is of use in proving the moment-cumulant formula.  

\begin{proposition}\label{Receq1}
Let $r_n(\mu)$ and $r_n(\mu, \nu)$ be the momotone cumulants and c-monotone cumulants respectively, and let $m_n(\mu,\nu, t)$ be the quantity 
defined in (\ref{cum18}). Then we have the recursive differential equations 
\begin{equation}\label{receq1}
\begin{split}
\frac{d}{dt}m_n(\mu, \nu, t) &= \sum_{k=0}^{n-1}(k+1)r_{n-k}(\nu)m_k(\mu, \nu, t) - \sum_{k = 0}^{n-1}\sum_{l=0}^{k} r_{n-k}(\nu)m_l(\mu, \nu, t)m_{k-l}(\mu, \nu, t) \\ &~~~~~~~~~~~+ \sum_{k = 0}^{n-1}\sum_{l=0}^{k} r_{n-k}(\mu, \nu)m_l(\mu, \nu, t)m_{k-l}(\mu, \nu, t). 
\end{split}
\end{equation}
\end{proposition}
\begin{remark}
In the special case $\mu = \nu$, the above equation becomes $\frac{d}{dt}m_n(\mu, t) = \sum_{k=0}^{n-1}(k+1)r_{n-k}(\mu)m_k(\mu, t)$ which has been obtained in the case of monotone convolutions \cite{H-S}. Moreover, we have 
\begin{equation} 
\frac{d}{dt}m_n(\mu, \delta_0, t) =  \sum_{k = 0}^{n-1}\sum_{l=0}^{k} r_{n-k}(\mu, \delta_0)m_l(\mu, \delta_0, t)m_{k-l}(\mu, \delta_0, t),  
\end{equation}
since $r_n(\delta_0)=0$. In this case $r_n(\mu, \delta_0)$ is an $n$-th Boolean cumulant. It might be interesting that this differential equation describes the structure of interval partitions. 
\end{remark}
\begin{proof}
We use the same notation as used in the subsection \ref{condmonocum1}. 
Differentiating the equality (\ref{eq50}) with $s$ and putting $s = 0$ we obtain 
\begin{equation}
\frac{\partial H_{(\mu, \nu)} }{\partial t} (t, z) = A_\nu(z)\frac{\partial H_{(\mu, \nu)} }{\partial z} (t, z) - A_\nu(z) + A_{(\mu, \nu)}(z). 
\end{equation}
We define $G_{(\mu, \nu)} (t, z) := \frac{1}{H_{(\mu, \nu)}(t, z)}$. 
Then we get the equality 
\begin{equation}\label{eq101}
\frac{\partial G_{(\mu, \nu)} }{\partial t} (t, z) = A_\nu(z)\frac{\partial G_{(\mu, \nu)} }{\partial z} (t, z) + G_{(\mu, \nu)} (t, z)^2 A_\nu(z) - G_{(\mu, \nu)} (t, z)^2 A_{(\mu, \nu)}(z). 
\end{equation}
(\ref{receq1}) follows from the expansion of (\ref{eq101}) in formal power series. 
\end{proof}

We say a block $V$ in a partition $\pi \in \pn$ is of \textit{interval type} if there exist $j$ and $k$ ($1 \leq j \leq n$, $0 \leq k \leq n - j$) such that $V = \{j, j+1, \cdots, j + k\}$.   
\begin{theorem}
The moment-cumulant formula for c-monotone independence is  
\begin{equation}\label{m-c11}
m_n(\mu) = \sum_{(\pi, \lambda) \in \mpn} \frac{1}{|\pi|!}\Big{(} \prod_{\substack{V \in \pi, \\ V: \text{\normalfont ~outer}}}r_{|V|}(\mu, \nu) \Big{)} \Big{(}  \prod_{\substack{V \in \pi, \\ V: \text{\normalfont ~inner}}}r_{|V|}(\nu) \Big{)}.
\end{equation}
\end{theorem}
\begin{proof} 
This proof gives the combinatorial meaning of the differential equations (\ref{receq1}).  
The claim is proved by induction on $n$: assume that the formula 
\begin{equation}
m_n(\mu,\nu,t) = \sum_{(\pi, \lambda) \in \mpn} \frac{t^{|\pi|}}{|\pi|!}\Big{(} \prod_{\substack{V \in \pi, \\ V: \text{\normalfont ~outer}}}r_{|V|}(\mu, \nu) \Big{)} \Big{(}  \prod_{\substack{V \in \pi, \\ V: \text{\normalfont ~inner}}}r_{|V|}(\nu) \Big{)}
\end{equation}
holds for $1 \leq n \leq N$, where $m_n(\mu, \nu, t)$ is defined in (\ref{cum18}).    
 Let $\pi = \{V_1 < \cdots < V_{|\pi|} \}$ denote a monotone partition $(\pi, \lambda) \in \mathcal{M}(N+1)$; this notation means that a subscript of a block itself 
expresses the linear ordering. Note that $V_{|\pi|}$ is always a block of interval type. Let $k$ be defined as $|V_{|\pi|}| = N+1-k$ ($0 \leq k \leq N$). 
Two cases are possible: (1) $V_{|\pi|}$ is outer; (2) $V_{|\pi|}$ is inner. 

(1) If $V_{|\pi|}$ is an outer block, $\pi$ is of such a form as 
\begin{center}
\includegraphics[width = 5cm]{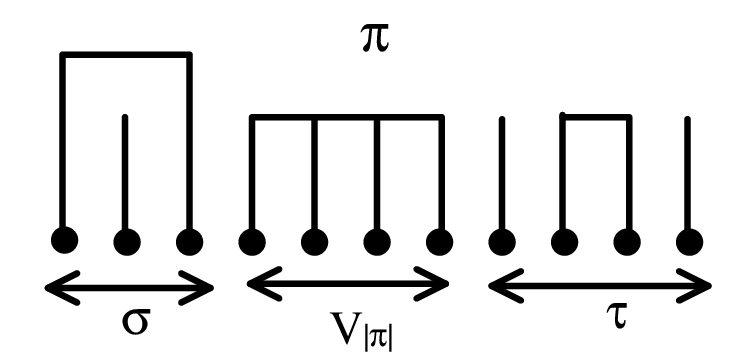} 
\end{center}
where $\sigma $ and $\tau$ are arbitrary non-crossing partitions with $\sigma \in \mathcal{NC}(l)$ and $\tau \in \mathcal{NC}(k-l)$ ($0 \leq l \leq k$). 
We understand that $\sigma = \emptyset$ for $l=0$ and $\tau = \emptyset$ for $l=k$. Next we need to consider the linear orderings of $\sigma$ and $\tau$. A linear ordering of the blocks of $\sigma \cup \tau$ can be described by distributing natural numbers $\{1, \cdots, |\sigma| + |\tau| \}$ to them such that $\sigma$ and $\tau$, equipped with the natural numbers, become monotone partitions. Once $\sigma$ and $\tau$ are fixed, there exist $\frac{(|\sigma| + |\tau|)!}{|\sigma|! |\tau|!}$ ways to divide the set 
$\{1, \cdots, |\sigma| + |\tau| \}$ into two subsets $C$ and $D$ with $|C| = |\sigma|$ and $|D| = |\tau|$. After such a division, we can choose linear orderings of $\sigma$ from $C$ and $\tau$ from $D$ independently with each other, and hence, we can take arbitrary 
$(\sigma,\rho) \in \mathcal{M}(l)$ and $(\tau,\mu) \in \mathcal{M}(k-l)$. The above arguments imply that 
\begin{equation}
\begin{split}
&\sum_{\substack{(\pi, \lambda) \in \mathcal{M}(N+1); \\ V_{|\pi|}: \text{~outer}}} \frac{t^{|\pi|}}{|\pi|!}\Big{(} \prod_{\substack{V \in \pi, \\ V:  \text{~outer}}}r_{|V|}(\mu, \nu) \Big{)} \Big{(}  \prod_{\substack{V \in \pi, \\ V: \text{~inner}}}r_{|V|}(\nu) \Big{)} \\
& = \sum_{k = 0} ^N \sum_{l = 0}^k \sum_{\substack{(\sigma, \rho) \in \mathcal{M}(l) \\ (\tau, \mu) \in \mathcal{M}(k-l)}} 
\frac{(|\sigma| + |\tau|)!}{|\sigma|! |\tau|!} \frac{t^{|\sigma| + |\tau| + 1} r_{N+1-k}(\mu, \nu)}{(|\sigma| + |\tau| + 1)!}  \Big{(} \prod_{\substack{V \in \sigma \cup \tau, \\ V: \text{~outer}}}r_{|V|}(\mu, \nu) \Big{)} \Big{(}  \prod_{ \substack{ V \in \sigma \cup \tau, \\ V: \text{~inner}}}r_{|V|}(\nu) \Big{)} \\
&=  \sum_{k = 0} ^N \sum_{l = 0}^k r_{N+1-k}(\mu, \nu) \int_0 ^t m_l(\mu, \nu, s)m_{k-l}(\mu, \nu, s) ds, 
\end{split}
\end{equation} 
where we have used the assumption of induction in the last line. 

(2) Next we consider the case where $V_{|\pi|}$ is inner. Since sums with $V_{|\pi|}$ inner blocks are difficult to compute directly, first we sum over all monotone partitions and later we remove the redundant sums, i.e., the sums over monotone partitions such that $V_{|\pi|}$ are outer. When $|V_{|\pi|}| = N+1-k$, there are $k + 1$ ways to choose a position of $V_{|\pi|}$ as a subset of $\{1,\cdots, N\}$. After the choice of $V_{|\pi|}$, we can take arbitrary monotone partition of $\mathcal{M}(k)$ as if the highest block $V_{|\pi|}$ were absent. The prototype of this idea appeared in \cite{Saigo} and was used in \cite{H-S}. The above arguments then amount to 
\begin{equation}
\begin{split}
&\sum_{\substack{(\pi, \lambda) \in \mathcal{M}(N+1); \\ V_{|\pi|}: \text{~inner}}} \frac{t^{|\pi|}}{|\pi|!}\Big{(} \prod_{\substack{V \in \pi, \\ V:  \text{~outer}}}r_{|V|}(\mu, \nu) \Big{)} \Big{(}  \prod_{\substack{V \in \pi, \\ V: \text{~inner}}}r_{|V|}(\nu) \Big{)} \\
& = \sum_{k = 0} ^N (k+1)r_{N+1-k}(\nu) \sum_{(\sigma, \rho) \in \mathcal{M}(k)}\frac{t^{|\sigma| + 1}}{(|\sigma| + 1)!} 
\Big{(} \prod_{\substack{V \in \sigma, \\ V: \text{~outer}}}r_{|V|}(\mu, \nu) \Big{)} \Big{(}  \prod_{\substack{V \in \sigma, \\ V:  \text{~inner}}}r_{|V|}(\nu) \Big{)} \\ 
&~~~~~- \sum_{k = 0} ^N \sum_{l = 0}^k \sum_{\substack{(\sigma, \rho) \in \mathcal{M}(l) \\ (\tau, \lambda) \in \mathcal{M}(k-l)}} 
\frac{(|\sigma| + |\tau|)!}{|\sigma|! |\tau|!} \frac{t^{|\sigma| + |\tau| + 1} r_{N+1-k}(\nu)}{(|\sigma| + |\tau| + 1)!}  \Big{(} \prod_{\substack{V \in \sigma \cup \tau, \\ V: \text{~outer}}}r_{|V|}(\mu, \nu) \Big{)} \Big{(}  \prod_{\substack{V \in \sigma \cup \tau, \\ V: \text{~inner}}}r_{|V|}(\nu) \Big{)} \\
&= \sum_{k = 0} ^N (k+1)r_{N+1-k}(\nu)\int_0 ^t m_k(\mu, \nu, s)ds - \sum_{k = 0} ^N \sum_{l = 0}^k r_{N+1-k}(\nu) \int_0 ^t m_l(\mu, \nu, s)m_{k-l}(\mu, \nu, s) ds.  
\end{split}
\end{equation}

Combining the results from (1), (2) and Proposition \ref{Receq1}, we have 
\begin{equation}
m_{N + 1}(\mu, \nu, t) = \sum_{(\pi, \lambda) \in \mathcal{M}(N+1)} \frac{t^{|\pi|}}{|\pi|!}\Big{(} \prod_{\substack{V \in \pi, \\ V: \text{~outer}}}r_{|V|}(\mu, \nu) \Big{)} \Big{(}  \prod_{\substack{V \in \pi, \\ V: \text{~inner}}}r_{|V|}(\nu) \Big{)}.        
\end{equation}
\end{proof}
\begin{example}
The moment-cumulant formula for $n = 1, 2, 3, 4$ is calculated as 
\begin{equation*}
\begin{split}
&m_1(\mu) = r_1(\mu, \nu), \\ 
&m_2(\mu) = r_2(\mu, \nu) + r_1(\mu, \nu)^2, \\
&m_3(\mu) = r_3(\mu, \nu) + 2 r_2(\mu, \nu)r_1(\mu, \nu) + \frac{1}{2}r_2(\mu, \nu)r_1(\nu) + r_1(\mu, \nu)^3, \\
&m_4(\mu) = r_4(\mu, \nu) + 2r_3(\mu, \nu) r_1(\mu, \nu) + r_3(\mu, \nu) r_1(\nu) + r_2 (\mu, \nu)^2 + \frac{1}{2} r_2(\mu, \nu) r_2(\nu) \\ &~~~~~~~~~~~+ 3 r_2(\mu, \nu) r_1(\mu, \nu)^2  + r_2(\mu, \nu) r_1(\mu, \nu) r_1(\nu) +  \frac{1}{3} r_2(\mu, \nu) r_1(\nu) ^2 + r_1(\mu, \nu) ^4. 
\end{split}
\end{equation*} 
\end{example}

\section{Limit theorems}\label{Lim}
We show the central limit theorem and Poisson's law of small numbers 
for c-monotone independence. The $t$-transformation $\mathcal{U}_t$ \cite{BW1}, defined by 
\[
H_{\mathcal{U}_t(\mu)}(z) = (1-t)z +t H_\mu(z), 
\] 
appears in the limit distributions. 
\begin{theorem}
Let $(\mathcal{A}, \varphi, \psi)$ be a $C^\ast$-algebraic probability space with two states. \\
(1) Let $\{X_i \}_{i = 1} ^{\infty}$ be identically distributed (with respect to each state), c-monotone independent self-adjoint random variables in $\mathcal{A}$. If $\varphi(X_i) = \psi(X_i)  = 0$, $\varphi(X_i) = \alpha^2$ and $\psi(X_i^2) = \beta^2$, then 
the distribution of $\frac{X_1 + \cdots + X_n}{\sqrt{n}}$ with respect to $(\varphi, \psi)$ converges to $(\mathcal{U}_{\alpha^2 / \beta ^2}(\nu_{\beta^2}), \nu_{\beta ^2})$ weakly, where $\nu_{\beta^2}$ is the arcsine law with variance $\beta^2$. $\mathcal{U}_{\alpha^2 / \beta ^2}(\nu_{\beta ^2})$ is a Kesten distribution. \\
(2)  Let $X_N ^{(n)}$ $(1 \leq n \leq N,~1 \leq N < \infty)$ be self-adjoint random variables such that 
\begin{itemize} 
\item[(a)] for each $N$, $X_N ^{(n)} (1 \leq n \leq N)$ are identically distributed with respect to each state, c-monotone independent self-adjoint random variables;  
\item[(b)] $N \varphi((X_N^{(1)})^k) \to \lambda > 0$ and $N \psi((X_N^{(1)})^k) \to \rho > 0$ as $N \to \infty$ for all $k \geq 1$.
\end{itemize} 
Then the distribution of $X_N := X_N^{(1)} + \cdots + X_N^{(N)}$ with respect to $(\varphi, \psi)$  converges weakly to $(\mathcal{U}_{\lambda / \rho}(p_{\rho}), p_{\rho})$, where $p_\rho$ is the monotone Poisson distribution with parameter $\rho$. 
\end{theorem}
\begin{proof}
We discuss the problems in terms of probability measures. 
Let $\mu, \nu \in \pc$ such that $m_1(\mu) = m_1(\nu) = 0$, $m_2(\mu) = \alpha ^2$ and $m_2(\nu) = \beta^2$ with 
$\alpha^2, \beta^2 > 0$. We define $(\mu_N, \nu_N):= (D_{\frac{1}{\sqrt{N}}}\mu, D_{\frac{1}{\sqrt{N}}}\nu)^{\rhd N}$. The convolution for the second component is the usual monotone convolution, so that the limit distribution $\nu_{\beta^2}$ exists and it is the centered arcsine law with variance $\beta^2$. A simple calculation shows that $r_1(\mu_N, \nu_N) = 0,~ r_2(\mu_N, \nu_N) = \alpha^2$ and $r_n(\mu_N, \nu_N) \to 0$ as $N \to \infty$. Therefore, there exists a pair of limit distributions $(\nu_{\alpha^2, \beta^2}, \nu_{\beta^2})$ at least in the sense of moments. The limit distributions are characterized by 
\begin{gather}
\frac{d}{dt}H_{(\nu_{\alpha^2, \beta^2},\nu_{\beta^2})}(t, z) = -\frac{\alpha^2}{H_{\nu_{\beta^2}}(t, z)},~ H_{(\nu_{\alpha^2, \beta^2},\nu_{\beta^2})}(1, z) = H_{\nu_{\alpha^2, \beta^2}}(z), \\ \frac{d}{dt}H_{\nu_{\beta^2}}(t, z) = -\frac{\beta^2}{H_{\nu_{\beta^2}}(t, z)},~ H_{\nu_{\beta^2}}(1 ,z) = H_{\nu_{\beta^2}}(z). \label{limit111}
\end{gather}
We obtain $H_{\nu_{\beta^2}}(t, z) = \sqrt{z^2 - 2\beta^2 t}$ from (\ref{limit111}). Setting $t = 1$, we obtain the limit measure $\nu_{\beta^2}(dx) = \frac{1}{\pi \sqrt{2\beta^2 - x^2}}dx$, $x \in (-\sqrt{2}\beta, \sqrt{2}\beta)$. 
For the first component, we can easily prove that $H_{(\nu_{\alpha^2, \beta^2}, \nu_{\beta^2})}(1, z) = (1 -\frac{\alpha^2}{\beta^2})z + \frac{\alpha^2}{\beta^2}\sqrt{z^2 - 2\beta^2}$, and hence 
\begin{equation}\label{Kesten1}
G_{\nu_{\alpha^2, \beta^2}}(z)= \frac{\alpha^2\sqrt{z^2 - 2\beta^2} + (\alpha^2 - \beta^2)z}{(2 \alpha^2 - \beta^2)z^2 - 2\alpha^4}.  
\end{equation}
This is the Stieltjes transform of a Kesten distribution. 
In terms of the $t$-transformation $\mathcal{U}_t$, the limit distribution can be written as $\mathcal{U}_{\alpha^2 / \beta^2}(\nu_{\beta^2})$. 
It is compactly supported, so that the convergence is in the sense of weak convergence (see Theorem 4.5.5 of \cite{KLC}). 
We note that $\nu_{\beta^2, \beta^2} = \nu_{\beta^2}$. 

Next we consider Poisson's law of small numbers.  
Let $\mu^{(N)}, \nu^{(N)} \in \pc$ such that $Nm_n(\mu^{(N)}) \to \lambda > 0$ and 
$Nm_n(\nu^{(N)}) \to \rho > 0$ as $N \to \infty$ for any $n \geq 1$. We define $(\mu_N, \nu_N):= (\mu^{(N)}, \nu^{(N)})^{\rhd N}$. 
It is not difficult to prove that $r_n(\mu_N, \nu_N) \to \lambda$ as $N \to \infty$ for any $n \geq 1$. 
It is known that $\nu_N$ converges weakly to the monotone Poisson distribution $p_\rho$ with parameter $\rho$, which is characterized by the differential equation 
\begin{equation}\label{diff11}
\frac{d}{dt}H_{p_\rho}(t, z) = \frac{\rho H_{p_\rho}(t, z)}{1-H_{p_\rho}(t, z)},~  H_{p_\rho}(1, z) = H_{p_\rho}(z), 
\end{equation} 
where $H_{p_\rho}(1, z) = H_{p_\rho}(z)$. 
We let $p_{\lambda, \rho}$ denote the limit distribution $\lim_{N \to \infty} \mu_{N}$ in the sense of moments. 
Then $p_{\lambda, \rho}$ is characterized by the differential equation 
\begin{equation}\label{diff12}
\frac{d}{dt}H_{(p_{\lambda, \rho}, p_\rho)}(t, z) = \frac{\lambda H_{p_\rho}(t, z)}{1-H_{p_\rho}(t, z)},~  H_{(p_{\lambda, \rho}, p_\rho)}(1, z)=H_{p_{\lambda, \rho}}(z). 
\end{equation} 
From  (\ref{diff11}) and (\ref{diff12}) we have $\frac{d}{dt}H_{(p_{\lambda, \rho}, p_\rho)}(t, z) = \frac{\lambda}{\rho}\frac{d}{dt}H_{p_\rho}(t, z)$, which implies 
\begin{equation}
H_{p_{\lambda, \rho}}(z) = \Big{(}1 - \frac{\lambda}{\rho} \Big{)}z + \frac{\lambda}{\rho}H_{p_\rho}(z). 
\end{equation}
It is not difficult to see that $p_{\lambda, \rho}$ has a compact support. Therefore, $\mu_N$ converges weakly to $p_{\lambda, \rho}$. Also in this case, $p_{\lambda, \rho}$ can be written as $p_{\lambda, \rho} = \mathcal{U}_{\lambda / \rho}(p_{\rho})$. 
\end{proof}

\section{Convolution semigroups} \label{Inf1}
We investigate convolution semigroups and infinite divisibility from this section. First we establish the equivalence between a pair of vector fields and a weakly continuous c-monotone convolution semigroup. We follow the method used by Muraki. We use the notation $H_t(z)$ and $F_t(z)$ for the reciprocal Cauchy transforms of the left component of a semigroup and of the right one, respectively. 
\begin{theorem}\label{thm1}
Let $\{F_t(z) \}_{t \geq 0}$ be a composition semigroup of reciprocal Cauchy transforms with $F_0(z) = z$. 
We assume that $F_t(z)$ is continuous with respect to $t \in [0, \infty)$ for each fixed $z \in \comp \backslash \real$. 
Let $\{H_t(z) \}_{t \geq 0}$ be a family of reciprocal Cauchy transforms with $H_0(z) = z$, satisfying 
the same continuity condition as $F_t(z)$ and satisfying the equality $H_{t+s}(z) = H_t(F_s(z)) - F_s(z) + H_s(z)$. Then 
there exist analytic vector fields $A_1$ and $A_2$ such that 
$H_t$ and $F_t$ satisfy the differential equations 
\begin{align}
\frac{d}{dt} H_t(z) = A_1(F_t(z)), \label{diff1} \\
\frac{d}{dt} F_t(z) = A_2(F_t(z)).  \label{diff2}
\end{align}
Moreover, the vector fields $A_1(z) = \frac{d}{dt}|_0 H_t(z)$ and $A_2(z) = \frac{d}{dt}|_0 F_t(z)$ have the representations 
\begin{align}\label{eq31}
& A_j(z) = -\gamma_j + \int_{\real} \frac{1 + xz}{x - z} d\tau_j(x) 
\end{align}
for $j = 1, 2$, where $\tau_j$ is a positive finite measure and $\gamma_j \in \real$ 
(This is the L\'{e}vy-Khintchine formula for the c-monotone convolution). 

Conversely, given analytic vector fields $A_1$ and $A_2$ on the upper halfplane of the forms (\ref{eq31}), by solving the equations (\ref{diff1}) and (\ref{diff2}) 
we obtain  $\{F_t(z) \}_{t \geq 0}$ and $\{H_t(z) \}_{t \geq 0}$ which are $C^\omega$ functions with respect to $(t, z)$ and satisfy $F_0(z) = z$,  $H_0(z) = z$, $F_{t + s} = F_t \circ F_s$ and $H_{t+s}(z) = H_t(F_s(z)) - F_s(z) + H_s(z)$.  
\end{theorem}
\begin{proof}
We use the representations  
\begin{align}\label{eq1}
&H_t (z) = a_t + z + \int_{\real} \frac{1 + xz}{x-z} d\eta_t(x), \\
&F_t (z) = b_t + z + \int_{\real} \frac{1 + xz}{x-z} d\xi_t(x). 
\end{align}
We follow the method of \cite{Mur3}. 
The claims for $F_t$ were proved in \cite{Mur3}; we only have to prove the claims for $H_t$. 
However, the proof below actually includes that for $F_t$ when we put $H_t = F_t$. 
A direct calculation leads to the following equality: 
\begin{equation}\label{eq1.1}
\begin{split}
H_{t + s}(z) - H_t(z) &= a_s +  \int_{\real}\frac{1 + xz}{x-z}\eta_s(dx) + (F_s(z) - z)\int_{\real} \frac{1 + x^2}{(x - F_s(z))(x - z)} \eta_t(dx) \\
                      &= (H_s (z) - z) +  (F_s(z) - z)\int_{\real} \frac{1 + x^2}{(x - F_s(z))(x - z)} \eta_t(dx). 
\end{split}
\end{equation}

\textit{Step (1): the right differentiability of $H_t$ and $F_t$ at $0$.}
Take $\delta > 0$, $n \in \nat$ and $k \in \nat$ ($0 \leq k \leq n$). 
We set $s = \frac{\delta}{n}$ and $t = \frac{k}{n}\delta$. Summing the equality (\ref{eq1.1}) over $k$, we obtain 
\begin{equation}\label{eq3}
H_{\delta}(z) - z = n(H_{\frac{\delta}{n}}(z) - z) + (F_{\frac{\delta}{n}}(z) - z)\sum_{k = 0} ^{n-1}\int_{\real} \frac{1 + x^2}{(x - F_{\frac{\delta}{n}}(z))(x - z)} \eta_{\frac{k}{n}\delta}(dx). \\ 
\end{equation}
Then 
 \begin{equation}\label{eq4}
\begin{split}
\frac{H_{\delta}(z) - z}{\delta} = \frac{H_{\frac{\delta}{n}}(z) - z}{\delta / n} + \delta^{-1}\cdot \frac{F_{\frac{\delta}{n}}(z) - z}{\delta / n} \cdot \frac{\delta}{n}\sum_{k = 0} ^{n-1}\int_{\real} \frac{1 + x^2}{(x - F_{\frac{\delta}{n}}(z))(x - z)} \eta_{\frac{k}{n}\delta}(dx). \\ 
\end{split}
\end{equation}
Since $H_t(i) = a_t + (1 + \eta_t(\real))i$, $a_t$ is a continuous function. If $F_s(z) = z$ for some $s > 0$, then $F_s(z) = z$ for all $s > 0$, so that the parameter $t$ of $H_t$ expresses the Boolean time evolution. In this case, the claims are trivial. Therefore, we assume that $F_s(z) \neq z$ for all $s > 0$. 
Then $\int_{\real} \frac{1 + x^2}{(x - F_s(z))(x - z)} \eta_t(dx)$ is a continuous function of $t$ since it can be expressed as  
$(H_t(F_s(z)) - H_t(z) - F_s(z)+z)/(F_s(z)-z)$. Therefore, 
 \begin{equation}\label{eq5}
\lim_{n \to \infty}\frac{\delta}{n}\sum_{k = 0} ^{n-1}\int_{\real} \frac{1 + x^2}{(x - F_{\frac{\delta}{n}}(z))(x - z)} \eta_{\frac{k}{n}\delta}(dx) =  \int_{0}^{\delta} dt \int_{\real} \frac{1 + x^2}{(x - z)^2} \eta_{t}(dx). 
\end{equation}
We note that $\lim_{t \searrow 0} \eta_t(\real) = 0$, and hence, for small enough $\delta > 0$ the integral 
 $\frac{1}{\delta}\int_{0}^{\delta} dt \int_{\real} \frac{1 + x^2}{(x - z)^2} \eta_{t}(dx)$ is small. 
If we put $H_t$ equal to $F_t$, (\ref{eq5}) implies the existence of 
\begin{equation}\label{eq5.1}
A_{\delta, 2}(z) := \lim_{n \to \infty} \frac{F_{\frac{\delta}{n}}(z) - z}{\delta / n} = \frac{(H_{\delta}(z) - z)/ \delta}{1 +  \frac{1}{\delta}\int_{0}^{\delta} dt \int_{\real} \frac{1 + x^2}{(x - z)^2} \eta_{t}(dx)}
\end{equation}
for small $\delta$, which has been obtained in \cite{Mur3}. 
The limit does not depend on $\delta > 0$ as shown in \cite{Mur3}; we summarize the proof here. We can prove easily that $A_{r\delta, 2} = A_{\delta, 2}$ for any positive rational number $r$. The equality (\ref{eq5.1}) implies that $A_{\delta, 2}$ depends on $\delta$ continuously. Therefore, $A_{r\delta,2} = A_{\delta, 2}$ for every positive real number $r > 0$. Then it may be simply denoted by $A_2(z)$. 

We note again that the equality 
\begin{equation}\label{eq6.0}
A_2(z) = \frac{(F_{\delta}(z) - z)/ \delta}{1 +  \frac{1}{\delta}\int_{0}^{\delta} dt \int_{\real} \frac{1 + x^2}{(x - z)^2} \eta_{t}(dx)}
\end{equation}
holds.  
Taking $\delta \to 0$ in the above equality we obtain $A_2(z) = \lim_{\delta \searrow 0} \frac{F_{\delta}(z) - z}{\delta}$.  
Again for general $(H_t, F_t)$, by taking the limit $n \to \infty$ in (\ref{eq4}) we have the existence of the limit 
$A_{1, \delta}(z) = \lim_{n \to \infty} \frac{H_{\frac{\delta}{n}}(z) - z}{\delta / n}$ which satisfies 
\begin{equation}\label{eq6}
\frac{H_{\delta}(z) - z}{\delta} = A_{1, \delta}(z) + A_2(z) \delta^{-1}\int_{0}^{\delta} dt \int_{\real} \frac{1 + x^2}{(x - z)^2} \eta_{t}(dx). \\ 
\end{equation} 
With the same argument as for $A_{2, \delta}$, $A_{1, \delta}$ does not depend on $\delta$, so that it may be denoted by $A_1$. Taking the limit $\delta \searrow 0$ in (\ref{eq6}), we have the right differentiability $\lim_{\delta \searrow 0} \frac{H_{\delta}(z) - z}{\delta} = A_1(z)$. 

\textit{Step (2): the differentiability of $H_t$ and $F_t$.} We do not refer to the claims for $F_t$, since it is known in \cite{Mur3}, or since the proof below is true for $F_t$ if we put $H_t = F_t$. 

We define the right and left derivatives by setting $D_t ^+ H_t(z) = \lim_{\delta \searrow 0}\frac{H_{t + \delta} (z) - H_t(z)}{\delta}$ 
and $D_t ^- H_t(z) = \lim_{\delta \searrow 0}\frac{H_{t - \delta} (z) - H_t(z)}{\delta}$ respectively. 
Then  
\begin{equation}\label{eq7}
\begin{split}
D_t ^+ H_t(z) &= \lim_{\delta \searrow 0}\frac{H_{\delta} (F_t(z)) - F_t(z)}{\delta} \\
              &=  A_1(F_t(z)). 
\end{split}
\end{equation}
Take $T > 0$, $\delta > 0$, $n \in \nat$ and $k \in \nat$ ($0 \leq k \leq n$). 
From an argument similar to the derivation of (\ref{eq4}), we have 
 \begin{equation}\label{eq8}
\begin{split}
\frac{H_{T}(z) - H_{T- \delta}(z)}{\delta} = \frac{H_{\frac{\delta}{n}}(z) - z}{\delta / n} + \delta^{-1}\cdot \frac{F_{\frac{\delta}{n}}(z) - z}{\delta / n} \cdot \frac{\delta}{n}\sum_{k = 1} ^{n}\int_{\real} \frac{1 + x^2}{(x - F_{\frac{\delta}{n}}(z))(x - z)} \eta_{T - \frac{k}{n}\delta}(dx). \\ 
\end{split}
\end{equation}
Taking the limit $n \to \infty$, we obtain 
\[
\frac{H_{T}(z) - H_{T- \delta}(z)}{\delta} = A_1(z) + \delta^{-1}A_2(z) \int_{0}^{\delta} dt \int_{\real} \frac{1 + x^2}{(x - z)^2} \eta_{T-t}(dx).
\]
Moreover, let $\delta \searrow 0$, to know that the limit 
\[
D_t ^- H_t(z) = A_1(z) + A_2(z)\Big{(}\frac{\partial H_t}{\partial z}(z) -1 \Big{)} ~~(t > 0) 
\]
exists. 
Finally, the differentiability of $H_t$ at $t > 0$ follows from the calculations 
 \begin{equation}\label{eq8.1}
\begin{split}
D_t ^+ H_t(z) &= \lim_{\delta \searrow 0}\frac{H_{t} (F_{\delta}(z)) - H_t(z) - F_\delta (z) + H_\delta (z)}{\delta} \\
              &= \lim_{\delta \searrow 0}\frac{H_{t} (F_{\delta}(z)) - H_t(z)}{F_\delta(z) - z} \cdot \frac{F_\delta(z) - z}{\delta} + \lim_{\delta \searrow 0}\frac{H_\delta(z) - z}{\delta} - \lim_{\delta \searrow 0}\frac{F_\delta(z) - z}{\delta} \\
              &= A_2(z)\frac{\partial H_t}{\partial z}(z) + A_1(z) - A_2(z) \\
              &= D_t ^- H_t(z).  
\end{split}
\end{equation}
(\ref{diff1}) follows from (\ref{eq7}). Then $H_t(z) = z + \int_0 ^t A_1(F_s(z))ds$, which implies 
that $H_t(z)$ is in $C^\omega ([0, \infty))$. 

\textit{Step (3): the representation of $A_1(z)$.} It is sufficient to prove that 
\begin{itemize}
\item[(i)] $A_1(z)$ is a Pick function, i.e., 
$A_1(z)$ is analytic in $\com+$ and maps $\com+$ into $\com+ \cup \real$; 
\item[(ii)] $\lim_{y \to \infty} \frac{\im A_1(iy)}{y} = 0$.  
\end{itemize}
Now we know the relation  
\begin{equation}\label{eq9}
A_{1}(z) = \frac{H_{\delta}(z) - z}{\delta} - A_2(z) \delta^{-1}\int_{0}^{\delta} dt \int_{\real} \frac{1 + x^2}{(x - z)^2} \eta_{t}(dx),  
\end{equation}  
from which $A_1$ is analytic. Since $\frac{H_\delta (z) - z}{\delta}$ is a Pick function, its limit $A_1$ is also. 
It is easy to prove the property (ii) by using (\ref{eq9}). 

\textit{Step (4): the converse statement.} We note that the solution of the differential equation 
\begin{equation}
\begin{split}
&\frac{d}{dt} F_t(z) = A_2(F_t(z)), \\
&F_0(z) = z,  
\end{split}
\end{equation}
does not explode in a finite time \cite{Be-Po}. Then the equation defines a flow of reciprocal Cauchy transforms indexed by the non-negative real numbers. We can define $H_t$ by setting
\begin{align}
H_t(z) = z + \int_0 ^t A_1(F_s(z))ds.   
\end{align}
We need to prove the equality $H_{t+s}(z) = H_t(F_s(z)) - F_s(z) + H_s(z)$. 
We fix $s \geq 0, z \in \com+$ and put $J_t(z):= H_{t+s}(z)$, $K_t(z):= H_t(F_s(z)) - F_s(z) + H_s(z)$. 
Then $\frac{d}{dt} J_t(z) = A_1(F_{t+s}(z))$ and $\frac{d}{dt}K_t(z) = A_1(F_t \circ F_s(z))$ with $J_0 = K_0$ by definition, so that  $\frac{d}{dt} J_t = \frac{d}{dt}K_t$. Therefore, $J_t = K_t$ for all $t \geq 0$. 
\end{proof}

\section{Constructions of convolution semigroups} 
Several examples of monotone convolution semigroups are known in \cite{Has3,Mur3}. In this section, we show several ways to construct c-monotone convolution semigroups from monotone and Boolean convolution semigroups. 

(1) Let $\{\mu_t \}_{ t \geq 0}$ be a weakly continuous Boolean convolution semigroup with $\mu_0 = \delta_0$. Then 
$\{(\mu_t, \delta_0) \}_{t \geq 0}$ is a c-monotone convolution semigroup. The vector fields $A_1$ and $A_2$ are given by 
\[
A_1(z) = -z + H_{\mu_1}(z), ~~A_2(z) = 0. 
\]

(2) Let $\{\mu_t \}_{ t \geq 0}$ be a weakly continuous monotone convolution semigroup with $\mu_0 = \delta_0$. 
Then $\{(\mu_t, \mu_t) \}_{t \geq 0}$ is a c-monotone convolution semigroup. The vector fields $A_1$ and $A_2$ coincide. 

(3) Let $\mathcal{U}_t$ be the $t$-transformation \cite{BW1}. We recall that $\mathcal{U}_t(\mu)$ is characterized by 
\[
H_{\mathcal{U}_t(\mu)}(z) = (1- t)z + tH_{\mu}(z). 
\]
Let $\{\mu_t \}_{ t \geq 0}$ be a weakly continuous monotone convolution semigroup with $\mu_0 = \delta_0$. 
Then $\{(\mathcal{U}_r(\mu_t), \mu_t) \}_{t \geq 0}$ is a c-monotone convolution semigroup for any $r \geq 0$ since, under the notation $H_t(z):= H_{\mathcal{U}_r (\mu_t)}$, we have
\begin{equation*}
\begin{split}
H_{t}(F_s(z)) - F_s(z) + H_s(z) &= (1 - r)F_s(z) + r F_t(F_s(z)) - F_s(z) + (1 - r)z + rF_s(z) \\
                                &= (1 - r)z + rF_{t + s}(z) \\ 
                                   &= H_{t+s}(z).     
\end{split}
\end{equation*}
In this case, $A_1(z) = \frac{d}{dt}|_{t = 0} H_t(z) = r A_2(z)$. We note that this construction includes the 
example (2) when $r = 1$. 
This property can be understood in terms of cumulants as follows. 
\begin{proposition} The relation
\begin{equation}
r_n(\mu^{\uplus t}, \mu) = t r_n  (\mu)  
\end{equation}
holds for all $\mu$ with the finite moments of all orders. 
\end{proposition} 
This means formally ``$r_n(\mu ^{\uplus t}, \mu) = r_n (\mu ^{\rhd t})$'', which connects the Boolean convolution with the monotone convolution. We will generalize this relation 
in Corollary \ref{rel1}.

(4) Let $\{\mu_t \}_{ t \geq 0}$ be a weakly continuous monotone convolution semigroup with $\mu_0 = \delta_0$ and $a \in \real$. 
Then $\{(\delta_{at}, \mu_t) \}_{t \geq 0}$ is a c-monotone convolution semigroup. The vector field $A_1$ is a constant $a$. 
  
(5) The Cauchy distribution $\lambda_{tb}(dx):= \frac{tb}{\pi(x^2 + (tb)^2)}dx$ is infinitely divisible in the tensor, free, Boolean and monotone cases and becomes a convolution semigroup with the time parameter $t \geq 0$. Moreover, $\lambda_{tb}$ is a strictly 1-stable distribution (see \cite{Be-Vo,Has3,Mur3,Sat,S-W}) for any one of the four kinds of convolutions. 
$\lambda_{tb}$ plays a special role also in the case of c-monotone independence; for instance, $\{(\lambda_{bt}, \mu_t) \}_{t \geq 0}$ is a c-monotone convolution semigroup for any monotone convolution semigroup $\{\mu_t \}_{ t \geq 0}$. 
 $A_1(z)$ is a constant $ib$. 

(6) We introduce an operation to construct a new c-monotone convolution semigroup for given c-monotone convolution semigroups. 
\begin{definition} A convolution $\kappa^{u, v}$: $\pa \times \pa \to \pa$ is defined by 
\begin{equation}
\kappa^{u, v} (\mu, \nu):= \mu^{\uplus u} \uplus \nu^{\uplus v} \text{~for~} u,v \geq 0.    
\end{equation}
In terms of reciprocal Cauchy transforms, we have
\begin{equation}
H_{\kappa^{u, v} (\mu, \nu)}(z) := u H_\mu (z) + v H_\nu (z) + (1 - u - v)z ~~ \text{for~~} u, v \geq 0.  
\end{equation}
\end{definition}
\begin{proposition} (1) Let $\{(\mu_t, \lambda_t) \}_{t \geq 0}$ and $\{(\nu_t, \lambda_t) \}_{t \geq 0}$ be c-monotone convolution semigroups with the same right component. We let $A_1 ^{\mu, \lambda}$ and $A_1 ^{\nu, \lambda}$ denote the vector fields for the left component of $\{(\mu_t, \nu_t) \}_{t \geq 0}$ and $\{(\nu_t, \lambda_t) \}_{t \geq 0}$, respectively, and also $A_2 ^\lambda$ denote the vector field for the common right component. 
Then $\{(\kappa ^{u, v}(\mu_t, \nu_t), \lambda_t) \}_{t \geq 0}$ is a c-monotone convolution semigroup.  
A pair of the vector fields is given by $(uA_1 ^{\mu, \lambda} + vA_1 ^{\nu, \lambda}, A_2 ^\lambda)$. \\ 
(2) In terms of cumulants, we can understand the convolution $\kappa^{u,v}$ in the form
\[
r_n(\kappa^{u, v}(\mu, \nu), \lambda) = ur_n(\mu, \lambda) + vr_n(\nu, \lambda) 
\]
for all probability measures $\mu, \nu, \lambda$ with finite moments of all orders. 
\end{proposition}
\begin{proof}
Denote by $H_t ^{\mu}$ and $H_t ^{\nu}$ the reciprocal Cauchy transforms of $\mu_t$ and $\nu_t$, respectively.
\begin{equation*}
\begin{split}
H_{\kappa ^{u, v}(\mu_t, \nu_t)}(F_s) - F_s + H_{\kappa ^{u, v}(\mu_s, \nu_s)} 
       &= u H_{t} ^{\mu}(F_s) + v H_{t} ^{\nu}(F_s) + (1 - u - v)F_s - F_s + u H_{s} ^{\mu} + v H_{t} ^{\nu} \\              &~~~+ (1 - u - v)z \\ 
       &= u(H_{t} ^{\mu}(F_s) - F_s + H_{s} ^{\mu}) + v(H_{t} ^{\nu}(F_s) - F_s + H_{s} ^{\nu}) +(1 - u - v)z  \\
       &= u H_{t+s} ^{\mu} + v H_{t+s} ^{\nu} +(1 - u - v)z \\    
       &= H_{\kappa^{u, v} (\mu_{t + s}, \nu_{t+s})}.  
\end{split}
\end{equation*}
The relation for the vector fields follows immediately. The second statement for cumulants follows from the definition of c-monotone cumulants.   
\end{proof}
We obtain a nontrivial property of the c-monotone cumulants. 
\begin{corollary}\label{rel1}
We have the following relation between the additive Boolean convolution and the c-monotone cumulants: 
\begin{equation}
r_n(\mu \uplus \nu, \lambda) = r_n(\mu, \lambda) + r_n(\nu, \lambda)  
\end{equation}
for all probability measures $\mu, \nu, \lambda$ with finite moments of all orders.
\end{corollary}
\begin{remark}
(1) Since $(\mu, \delta_0) \rhd (\nu, \delta_0) = (\mu \uplus \nu, \delta_0)$ and $r_n(\mu, \delta_0)$ is the $n$-th Boolean cumulant, this corollary is trivial for $\lambda = \delta_0$. \\
(2) 
The uniqueness of Boolean cumulants follows from axioms (C1) and (C2') (see Section \ref{Cum}). In the present case, $\{r_n(\cdot, \lambda) \}_{n \geq 1}$ satisfies (C1) and (C2), but do not satisfy (C2') for any $\lambda \neq \delta_0$. 
\end{remark}

\begin{example} With the above methods (1)-(6) and with examples of Boolean or monotone convolution semigroups in the literature, we can construct many examples of c-monotone convolution semigroups. We give an important example among them. 
Consider the Kesten distribution $\mu_{\sigma^2, r}$ characterized by
\begin{equation}\label{Kesten1}
G_{\mu_{\sigma ^2, r}}(z) := \frac{\sqrt{z^2 - 2\sigma^2} + (1 - r)z}{(2  - r)z^2 - 2\sigma^2 r}.  
\end{equation}
$\mu_{\sigma^2, 1}(dx) = \frac{1}{\pi\sqrt{2\sigma^2 - x^2}}dx$ is the centered arcsine law with variance $\sigma^2$. 
We note $\mathcal{U}_{r}(\mu_{\sigma^2, 1}) = \mu_{\sigma^2, r}$. Then $\{(\mu_{t, r}, \mu_{t, 1}) \}_{t \geq 0}$ is a c-monotone convolution semigroup 
since $\{ \mu_{t, 1} \}_{t \geq 0}$ is a monotone convolution semigroup. This example is connected with the central limit measure for the c-monotone convolution (see Section \ref{Lim}). 
\end{example}
In this section, we have shown constructions of c-monotone convolution semigroups and properties of c-monotone cumulants.  
By the way, we can prove similar results in the conditionally free case. Among them, the following is interesting. 
We omit the proof.  
\begin{proposition}\label{rel2}
Let $R_n(\mu, \nu)$ be the c-free cumulants of $(\mu, \nu)$. Then the identity 
\begin{equation}
R_n(\mu \uplus \nu, \lambda) = R_n(\mu, \lambda) + R_n(\nu, \lambda)
\end{equation}
holds. 
\end{proposition} 
Finally we obtain the following relation between $r_n(\mu, \nu)$ and $R_n(\mu, \nu)$.  We note that we need a result from Section \ref{Cum}. 
\begin{theorem}
There exist polynomials $P_{n,k}$ of $n- k$ variables for $2 \leq k \leq n-1$, $n \geq 3$ such that 
\begin{equation}
r_n (\mu, \nu) = R_n(\mu, \nu) + \sum_{k=2}^{n-1} P_{n,k}(m_1(\nu), \cdots, m_{n-k}(\nu)) R_k(\mu, \nu). 
\end{equation} 
We note that $r_1(\mu, \nu) = R_1(\mu, \nu) = m_1(\mu)$ and $r_2(\mu, \nu) = R_2(\mu, \nu) = m_2(\mu) - m_1 (\mu)^2$ for $n = 1,2$. 
Roughly, $r_n(\mu, \nu)$ is expressed by a linear combination of $R_n(\mu, \nu)$ with  polynomial coefficients $m_k(\nu)$. 
\end{theorem}
\begin{proof}
The existence of $P_{n,k}$ follows from Corollary \ref{rel1}, Proposition \ref{rel2} and the argument used in Proposition \ref{trans2} (1). 
We replace $(\mu, \nu)$ with $(D_\lambda \mu, D_\lambda \nu)$ and compare the powers of $\lambda$, to conclude that $P_{n, k}$ only depends on $m_1(\nu), \cdots, m_{n-k}(\nu)$.  The reason why $R_1(\mu, \nu)$ does not appear is that a term of the form $Q_n(m_1(\nu), \cdots, m_p(\nu))m_1(\mu)$ never appears in the moment-cumulant 
formulae in both c-monotone and c-free cases, except for the first order $r_1(\mu, \nu) = R_1(\mu, \nu) = m_1(\mu)$.  
\end{proof}
\begin{example}
The third cumulants for c-monotone and c-free cases are given by
\begin{equation*}
\begin{split}
&r_3(\mu, \nu) = m_3(\mu) - 2m_2(\mu)m_1(\mu) -\frac{1}{2}m_1(\nu)(m_2(\mu) - m_1(\mu)^2) +m_1(\mu)^3, \\
&R_3(\mu, \nu) = m_3(\mu) - 2m_2(\mu)m_1(\mu) -m_1(\nu)(m_2(\mu) - m_1(\mu)^2) +m_1(\mu)^3. 
\end{split}
\end{equation*} 
Therefore,  
\begin{equation}
r_3(\mu, \nu) = R_3(\mu, \nu) + \frac{1}{2}m_1(\nu)R_2(\mu, \nu). 
\end{equation}
\end{example}


\section{Infinite divisibility}\label{Inf4}
In this section, we define infinite divisibility for the c-monotone convolution and 
we characterize the infinite divisible distributions with compact supports. 
\begin{definition}
A pair of probability measures $(\mu, \nu)$ on $\real$ is said to be (additively) c-monotone infinitely divisible if for any $n \geq 1$ there exists a pair of probability measures $(\mu_n, \nu_n)$ such that $(\mu, \nu) = (\mu_n, \nu_n)^{\rhd n}$.  
\end{definition}

The above definition might become easy to understand in terms of random variables in a $C^*$-algebra: 
\begin{definition}
In a $C^\ast$-algebraic probability space $(\mathcal{A}, \varphi, \psi)$ equipped with two states, we say that a self-adjoint operator 
$X$ has a c-monotone infinitely divisible distribution if for any $n \geq 1$ there exist a $C^\ast$-algebraic probability space $(\mathcal{A}_n, \varphi_n, \psi_n)$ and identically distributed, c-monotone independent random variables $X_1, \cdots, X_n \in \mathcal{A}_n$ such that $X$ has the same distribution as $X_1 + \cdots + X_n$ with respect to the two states. 
\end{definition}
 
These definitions are the same for compactly supported probability measures, since we know a canonical realization of 
c-monotone independence in Section \ref{Condi} and since $\mu_n, \nu_n$ are compactly supported whenever $\mu, \nu$ are (see Lemma \ref{compact1}).   
In this paper, we focus on the convolution of probability measures, and hence, use the former definition. 

\begin{lemma}\label{compact1} 
Let $(\mu_3, \nu_3)$ be the c-monotone convolution of $(\mu_1, \nu_1)$ and $(\mu_2, \nu_2)$. \\
(1) Let $(a_i, \eta_i)$ and $(b_i, \xi_i)$ respectively denote the pairs of real numbers and finite measures appearing in the representations in (\ref{nev}) for $\mu_i$ and $\nu_i$. 
Then $\supp \eta_2 \subset \supp \eta_3$ and  $\supp \xi_2 \subset \supp \xi_3$. \\
(2) If $\mu_i$ and $\nu_i$ $(i = 1, 2)$ are compactly supported, also $\mu_3$ and $\nu_3$ are.  
\end{lemma}
\begin{proof}
(1) We only prove the claim for the first component since the fact for the second component is know in \cite{Mur3}. 
From a simple calculation, we obtain 
\begin{equation}
H_{\mu_3}(z) = a_{\mu_1} + a_{\mu_2} + z +  \int_{\real} \frac{1 + xH_{\nu_2}(z)}{x-H_{\nu_2}(z)} d\eta_1(x) + 
\int_{\real} \frac{1 + xz}{x-z} d\eta_2(x).  
\end{equation}
Applying the Stieltjes inversion formula, we have $\lim_{v \searrow 0} \int_a ^b \im H_{\mu_3}(u + iv)du = 0$ whenever 
$[a, b] \cap \supp \eta_3 = \emptyset$, which implies that $\supp \eta_2 \subset \supp \eta_3$. 

(2) If $G_{\mu_i}$ and $G_{\nu_i}$ ($i = 1, 2$) are analytic outside a ball, then $G_{\mu_1 \rhd_{\nu_2} \mu_2}$ and $G_{\nu_1 \rhd \nu_2}$ are also analytic outside a ball. 
\end{proof}

\begin{corollary} 
Let $\mu, \nu$ be probability measures with compact supports. An $n$-th root of $(\mu, \nu)$ for the c-monotone convolution is unique for any $n \geq 1$. 
\end{corollary}
\begin{proof}
Let $(\mu_n, \nu_n)$ be an $n$-th roof of $(\mu, \nu)$, i.e., $(\mu, \nu) = (\mu_n, \nu_n) ^{\rhd n}$. 
$\mu_n$ and $\nu_n$ are compactly supported from Lemma \ref{compact1}. 
With the power additivity of the monotone cumulants and the c-monotone cumulants, we have 
$r_k(\mu_n, \nu_n) = \frac{1}{n}r_k(\mu, \nu)$ and $r_k^M(\mu_n) = \frac{1}{n}r_k^M(\nu)$ for all $k \geq 1$. 
This implies the uniqueness. 
\end{proof}

We prove the c-monotone analogue for Theorem 13.6 of \cite{N-S1}. 
\begin{theorem}\label{correspondence}
Let $\mu, \nu$ be probability measures with compact support. 
The following statements are equivalent. 
\begin{itemize}
\item[(1)] $(\mu, \nu)$ is c-monotone infinitely divisible. 
\item[(2)] There exists a compactly supported, weakly continuous c-monotone convolution semigroup $\{(\mu_t, \nu_t) \}_{t \geq 0}$ with $(\mu_0, \nu_0) = (\delta_0, \delta_0)$ such that $(\mu_1, \nu_1) = (\mu, \nu)$. 
\item[(3)] Both $\{r_n(\mu, \nu) \}_{n \geq 2}$ and $\{r_n(\nu) \}_{n \geq 2}$ are positive definite sequences.  
\item[(4)] There exist compactly supported probability measures $\mu_N$, $\nu_N$ for each $N$ such that $(\mu_N, \nu_N)^{\rhd N}$ converges to $(\mu, \nu)$ weakly. 
\end{itemize} 
\end{theorem}
\begin{proof}
The implications $(2) \Rightarrow (1) \Rightarrow (4)$ follow from Lemma \ref{compact1}. Now we prove $(4) \Rightarrow (3)$. 
We note that 
\begin{align}
&r_n(\nu) = \lim_{N \to \infty} N \int_{\real} x^n \nu_N(dx), \\ 
&r_n(\mu, \nu) = \lim_{N \to \infty} N \int_{\real} x^n \mu_N(dx). 
\end{align}
For $a_1, \cdots, a_n \in \comp$, we have 
\begin{equation}\label{eq10}
\begin{split}
\sum_{j,k = 1}^{n}a_j\bar{a_k}r_{k+j}(\mu, \nu) &= \lim_{N \to \infty} N \sum_{j,k = 1}^{n}a_j\bar{a_k} \int_{\real} x^{k + j} \mu_N(dx) \\
               &= \lim_{N \to \infty} N \Big{|}\sum_{j = 1}^{n}a_j\int_{\real} x^{j} \mu_N(dx)\Big{|}^2 \\ 
               &\geq 0. 
\end{split}
\end{equation}
Next we prove the implication $(3) \Rightarrow (2)$. 
We do not know a priori the existence of $R > 0$ such that 
$|r_n(\mu, \nu)| \leq R^n$ and $|r_n(\nu)| \le R^n$. 
At least, however, there exist positive finite measures $\tau_1, \tau_2$ such that 
\begin{align}
&A_1 (z) := - r_1(\mu, \nu) + \int_{\real}\frac{1 + xz}{x - z} \tau_1(dx) = -\sum_{n = 1} ^{\infty}\frac{r_n(\mu, \nu)}{z^{n-1}}, \\
&A_2 (z) := - r_1(\nu) + \int_{\real}\frac{1 + xz}{x - z} \tau_2(dx) = -\sum_{n = 1} ^{\infty}\frac{r_n(\nu)}{z^{n-1}}
\end{align}
in the sense of asymptotic expansion. 
We define two functions $H_t$, $F_t$ and a weakly continuous c-monotone convolution semigroup $\{ (\mu_t, \nu_t) \}_{t \geq 0}$ using Theorem \ref{thm1}. We obtain $r_n(\mu_t, \nu_t) = tr_n(\mu, \nu)$ and $r_n(\nu_t) = tr_n(\nu)$ from the power additivity of the monotone cumulants and the c-monotone cumulants. 
Therefore, $\mu_1$ and $\nu_1$ have the same moments as $\mu$ and $\nu$, respectively. Since $\mu$ and $\nu$ are compactly supported, $(\mu_1, \nu_1) = (\mu, \nu)$; moreover, $\eta_1$ and $\xi_1$ are compactly supported. From Lemma \ref{compact1} (1), $\eta_t$ and $\xi_t$ are compactly supported for $0 \leq t \leq 1$. Recalling the equalities (\ref{eq6.0}) and (\ref{eq9}), we conclude that 
$\supp \tau_1 \subset \supp \eta_{\delta}$ and $\supp \tau_2 \subset \supp \xi_{\delta}$ for sufficiently small 
$\delta > 0$. Therefore, $\supp \tau_j$ is compact.  
It is immediate that $\mu_t$ and $\nu_t$ are compactly supported for all $0 \leq t < \infty$ from Lemma \ref{compact1} (2). 
\end{proof}

\section{Remarks and discussions on infinite divisibility}\label{discuss} 
Let $\disc$ be the unit disc in the complex plane. 
Theorem 1.1 of \cite{Be-Po} says that if a semigroup of analytic maps $\phi_t(z)$ ($t \geq 0$) defined on $\disc$ with $\phi_0(z) = z$ is continuous in $[0, \infty) \times \disc$, there exists an analytic vector field as a generator of the semigroup. As a result, the semigroup also belongs to $ C^\omega([0, \infty) \times \disc)$.  
There is a different proof of Theorem \ref{thm1} based on the above (and its generalization). 
This approach was used by Bercovici in \cite{Ber1} for multiplicative monotone convolutions (see also Franz's argument in \cite{Fra}). 

Weak convergence of probability measures on the real line is equivalent to pointwise convergence of the reciprocal Cauchy transforms as shown in \cite{Maa}. On the unit circle, the weak convergence is equivalent to pointwise convergence of $\eta_{\mu}(z):= 1 - \frac{z}{G_\mu (\frac{1}{z})}$, $z \in \disc$. This fact can be proved in the same idea as \cite{Maa}. 
Therefore, a given weakly continuous convolution semigroup $\mu_t$ on $\real$ (resp. on $\tor$) with $\mu_0 = \delta_0$ (resp.  $\mu_0 = \delta_1$) has a continuous $H_{\mu_t}(z)$ (resp. $\eta_{\mu_t}(z)$) for each $z$. The following fact is needed to apply Theorem 1.1 of \cite{Be-Po}.  
\begin{proposition}\label{prop10}
Let $\{\phi_t \}_{t \in I}$ be a family of analytic maps on $\disc$ parametrized by $t \in I$, where $I$ is an interval. We assume that the map $t \mapsto \phi_t(z)$ 
is continuous for each $z \in \disc$. Then the map $\phi: I \times \disc \to \disc$ defined by $\phi(t,z) = \phi_t(z)$ 
is continuous.  
\end{proposition}
The proof is not difficult; we will give a proof in \cite{Has4}.  
Since $\disc$ is analytically isomorphic to $\com+$, we can apply this to both additive and multiplicative convolutions.  

In the case of the c-monotone convolution, the functional relation for reciprocal Cauchy transforms is not only a composition semigroup: $F_{s + t} = F_s \circ F_t$ and $H_{t+s} = H_t \circ F_s - F_s + H_s$. We used Muraki's method to prove Theorem \ref{thm1}, but it is also possible to use a  method similar to Theorem 1.1 in \cite{Be-Po} for  
$(H_t,F_t)$. This method is useful especially for the multiplicative convolution and we will show it in \cite{Has4}.  

We proved in Theorem \ref{correspondence} the equivalence between infinite divisibility and the embedding of a measure into a convolution semigroup for compactly supported probability measures.  We also proved the positive definiteness of cumulants. This result 
is new even in the monotone case.


\section{Convolutions arising from the conditionally monotone convolution}\label{deform11}
The c-monotone convolution unifies the monotone and Boolean convolutions: 
\begin{gather}
(\mu, \mu) \rhd (\nu, \nu) = (\mu \rhd \nu, \mu \rhd \nu), \\
(\mu, \delta_0) \rhd (\nu, \delta_0) = (\mu \uplus \nu, \delta_0). 
\end{gather}
These are analogous to (\ref{free1}) and (\ref{boolean1}). We generalize these relations in this section 
in analogy with the c-free case. Therefore, we first explain the c-free case.

For a map $T$: $\pa \to \pa$, a new convolution $\boxplus_T$ can be defined by 
\begin{equation}\label{Bok}
(\mu \boxplus _T \nu, T\mu \boxplus T\nu) = (\mu, T\mu) \boxplus (\nu, T\nu). 
\end{equation}
A problem of Bo\.{z}ejko is to find all maps $T$: $\pa \to \pa$ such that  
\begin{equation}\label{Bo}
T(\mu \boxplus _T \nu) = T\mu \boxplus T\nu. 
\end{equation}
This relation exactly says that the graph $\{(\mu, T\mu); \mu \in \pa \}$ is closed under the c-free convolution. 
Once such $T$ is found, the new convolution is associative and commutative.
Moreover, this generalizes the free and Boolean convolutions.
Many maps satisfying (\ref{Bo}) were found in \cite{BW1,BW2,KW,KY,O1,O2}. 
 
Motivated by these works, we consider the following type of convolution: 
\begin{equation}\label{Ha}
(\mu \rt \nu, T\nu) = (\mu, \delta_0) \boxplus (\nu, T\nu). 
\end{equation} 
This relation is parallel to (\ref{Bok}) in terms of the c-monotone convolution: 
\begin{equation}\label{Ha}
(\mu \rt \nu, T\mu \rhd T\nu) = (\mu, T\mu) \rhd (\nu, T\nu). 
\end{equation} 
Clearly, this convolution includes Boolean and monotone convolutions 
if we take $T$ as $T\mu = \delta_0$ for all $\mu$ and $T = \text{Id}$, respectively. 
Now we characterize the associativity of the convolution $\rt$. 
\begin{proposition}\label{associativity}
$(1)$ $\rt$ is characterized by the equality 
\begin{equation}\label{aa}
H_{\mu \rt \nu} = H_\mu \circ H_{T\nu} + H_{\nu} - H_{T\nu}.
\end{equation}
$(2)$ The convolution $\rt$ becomes associative if and only if 
\begin{equation}\label{asso}
T(\mu \rt \nu) = T\mu \rhd  T\nu 
\end{equation}
for all $\mu$ and $\nu$. 
\end{proposition}
\begin{remark}
(1) In many cases $T$ is only defined in a subset of $\pa$ such as $\pmo$. In such a case, the above Proposition holds if the subset is closed under the convolution $\rt$. Henceforth, we often state results about the convolution $\rt$ only for $T$: $\pa \to \pa$ if such a generalization of the domain of $T$ is trivial.  \\
(2) It seems to be not known whether the condition (\ref{Bo}) is a necessary condition for the associativity of $\boxplus_T$. 
\end{remark}
\begin{proof}
(1) The proof is easy. \\
(2) First we calculate $H_{(\mu \rt \nu) \rt \lambda}$ as  
\begin{equation}\label{c}
\begin{split}
H_{(\mu \rt \nu) \rt \lambda} &= H_{\mu \rt \nu} \circ H_{T\lambda} + H_\lambda - H_{T\lambda} \\
         &= H_\mu \circ H_{T\nu} \circ H_{T\lambda} + H_\nu \circ H_{T\lambda} -H_{T\nu}\circ H_{T\lambda} + H_\lambda - H_{T\lambda}. 
\end{split}
\end{equation}
$ H_{\mu \rt ( \nu \rt \lambda)}$ is calculated as follows. 
\begin{equation}\label{d}
\begin{split}
 H_{\mu \rt ( \nu \rt \lambda)} &= H_\mu \circ H_{T(\nu \rt \lambda)} + H_{\nu \rt \lambda} - H_{T(\nu \rt \lambda)} \\
      &= H_\mu \circ H_{T(\nu \rt \lambda)} +  H_\nu \circ H_{T\lambda} + H_{\lambda} - H_{T\lambda} - H_{T(\nu \rt \lambda)}. 
\end{split}
\end{equation}
Then the associativity of the convolution implies that $ H_\mu \circ H_{T\nu} \circ H_{T\lambda} - H_\mu \circ H_{T(\nu \rt \lambda)}$ only depends on $\nu $ and $\lambda$. If $H_{T\nu} \circ H_{T\lambda}$ were not equal to $H_{T(\nu \rt \lambda)}$
for some $\nu$ and $\lambda$, then there would exist $w \in \comp \backslash \real$ 
such that $ H_{T\nu} \circ H_{T\lambda}(w) \neq H_{T(\nu \rt \lambda)} (w)$. 
In this case, $H_\mu \circ H_{T\nu} \circ H_{T\lambda}(w) - H_\mu \circ H_{T(\nu \rt \lambda)}(w)$ clearly depends on $\mu$, which is a contradiction. Therefore, we conclude that
\begin{equation}\label{e}
H_{T\nu} \circ H_{T\lambda} = H_{T(\nu \rt \lambda)}
\end{equation}
for all $\nu, \lambda \in \pa$.  
Conversely, if (\ref{e}) holds, it is not difficult to see that the convolution is associative. 
\end{proof}

We show the additivity of mean and variance; this will be used in the proof of Theorem \ref{gene}. 
\begin{proposition}\label{addi}
Let $T$: $\pa \to \pa$ be an arbitrary map. Then we have the following properties. 
\begin{itemize}
\item[$(1)$] $\pb$ is closed under the convolution $\rt$.  
\item[$(2)$] $m(\mu)$ and $\sigma^2(\mu)$ are additive with respect to the convolution $\rt$ considered in $\pb$: 
\begin{gather}
m(\mu \rt \nu) = m(\mu) + m(\nu), \\ 
\sigma^2 (\mu \rt \nu) = \sigma^2(\mu) + \sigma^2(\nu). 
\end{gather}
\end{itemize}
 \end{proposition}
\begin{proof}
We use the notation 
\begin{gather}
H_\mu (z) = -m(\mu) + z + \int \frac{1}{x-z}d\rho_\mu (x).
\end{gather}
Then we have
\begin{equation}\label{eq1}
\begin{split}
H_{\mu \rt \nu} (z) &= H_{\mu} \circ H_{T\nu} (z) + H_\nu (z) - H_{T\nu}(z) \\
                             &= -m(\mu) + H_{T\nu} (z) + \int_{\real} \frac{1}{x-H_{T\nu} (z)} d\rho_\mu(x) 
                                    +H_\nu (z) - H_{T \nu}(z) \\
                             &= -m(\mu) + H_\nu(z) + \int_{\real} \frac{1}{x-z} d\Big{(}\int_{\real}(T\nu)_y(x)d\rho_\mu(y) \Big{)} \\       
                                &= -m(\mu) - m(\nu) + z + \int_{\real} \frac{1}{x-z} d\Big{(}\int_{\real} (T\nu)_y(x)d\rho_\mu(y) + \rho_\nu(x) \Big{)}, \\ 
\end{split}
\end{equation}
where $\lambda_y \in \pa$ is defined by $H_{\lambda_y} = H_\lambda - y$ for $\lambda \in \pa$. 
With Lemma \ref{lem2211}, $\pb$ is closed under the convolution; moreover,   
mean and variance are additive.  
\end{proof}

Thus, $\pb$ is closed under the convolution $\rt$ for arbitrary map $T$. 
On the contrary, $\p+$ and $\ps$ are not closed in general under the convolution 
$\rt$. We show the necessary and sufficient conditions for $\p+$ and $\ps$.
\begin{proposition}\label{positive2} 
$(1)$ The following two conditions are equivalent. 
\begin{itemize}
\item[$(1a)$] $T(\p+) \subset \p+$,  
\item[$(1b)$] $\mu \rt \nu \in \p+$ for all $\mu$, $\nu \in \p+$. 
\end{itemize}
$(2)$ The following two conditions are equivalent. 
\begin{itemize}
\item[$(2a)$] $T(\ps) \subset \ps$, 
\item[$(2b)$] $\mu \rt \nu \in \ps$ for all $\mu$, $\nu \in \ps$.
\end{itemize}
\end{proposition}  
\begin{proof}
(1) We assume $(1a)$. From Lemma \ref{positive1}, $H_{\lambda}$ is analytic in $\comp \backslash [0, \infty)$ 
and $H_\lambda < 0$ in $(-\infty, 0)$ for $\lambda = \mu, \nu, T\nu$. 
Then the composition $H_\mu \circ H_{T\nu}$ is analytic in $\comp \backslash [0, \infty)$, and hence, 
$H_{\mu \rt \nu}$ is also analytic in the same region. This implies the condition $(1)$ in Lemma \ref{positive1}. 
We use the notation 
\begin{equation}
H_\mu (z) = b_\mu + z + \int_{\real} \frac{1 + xz}{x-z} d\eta_\mu(x) 
\end{equation}
and similarly for $H_\nu$. 
We put $g_\mu (z) := \int_{\real} \frac{1 + xz}{x-z} d\eta_\mu(x)$.  Proposition \ref{associativity} implies that 
\begin{equation}
\begin{split}
H_{\mu \rt \nu}(z) &= H_\mu \circ H_{T\nu}(z) + H_{\nu}(z) - H_{T\nu}(z) \\
                &= b_\mu + g_\mu (H_{T\nu}(z)) + H_\nu (z).  
\end{split}
\end{equation}
Since $g_\mu$ and $H_{T\nu}$ are non-decreasing, $g_\mu \circ H_{T\nu}$ is also non-decreasing. 
Then $b_\mu + g_\mu \circ H_{T\nu} (-0) \leq b_\mu + g_\mu (-0) = H_{\mu}(-0) \leq 0$. 
Therefore, we obtain $H_{\mu \rt \nu}(-0) \leq 0$.   

Next we assume $(1b)$. We shall prove the fact by \textit{reductio ad absurdum}; we assume that there exists $\nu \in \p+$ such that $T\nu \notin \p+$. 
In the notation (\ref{eq1}), we have 
\begin{equation}\label{eq2}
\begin{split}
H_{\mu \rt \nu}(z) &= H_\mu \circ H_{T\nu}(z) + H_{\nu}(z) - H_{T\nu}(z) \\
                   &=  -m(\mu) + H_\nu(z) + \int_{\real} \frac{1}{x-z} d\Big{(}\int_{\real}(T\nu)_y(x)d\rho_\mu(y) \Big{)} \\ 
                   &= -m(\mu) + H_\nu(z) - G_{\rho_\mu \rhd T\nu}(z) \\
\end{split}
\end{equation}
for all $\mu \in \pb$, where we defined $\rho_\mu \rhd T\nu$ using
the affinity of the left component of the monotone convolution.    
We can construct $\mu \in \p+$ such that  
$a(\rho_\mu) = 0$. Therefore,  we have 
$a(\rho_\mu \rt T\nu) \leq a(T\nu) < 0$ from Lemma \ref{estimate} (3); this inequality means that $G_{\rho_\mu \rhd T\nu}$ is not analytic in $\comp \backslash [0, \infty)$. 
On the contrary, both $H_{\mu \rt \nu}(z)$ and  $H_\nu(z)$ are 
analytic in $\comp \backslash [0, \infty)$ by assumption; this is a contradiction (we note that Lemma \ref{estimate} is applicable to all positive finite measures).  

(2) We assume $(2a)$. Since a probability measure $\mu$ is symmetric if and only if $H_\mu (-z) = - H_\mu(z)$ for $z \in \comp \backslash \real$, the proof is not difficult. 

Conversely, we assume $(2b)$. We take $\mu$ to be the arcsine law with mean 0 and variance $1$. Clearly, $\mu \in \ps$. (2b) implies that 
$H_{\mu \rt \nu}(-z)= -H_{\mu \rt \nu}(z)$ for all $\nu \in \ps$. Using (\ref{aa}) we have  
\begin{equation}
\sqrt{H_{T\nu}(z)^2 - 2} - \sqrt{H_{T\nu}(-z)^2 -2} = H_{T\nu}(z) + H_{T\nu}(-z) 
\end{equation}
for $z \in \com+$. 
After some calculations we obtain $H_{T\nu}(-z) = - H_{T\nu}(z)$, which means $T\nu \in \ps$.   
\end{proof}
\begin{remark} 
The above property unifies the properties of Boolean and monotone convolutions \cite{Has2}. 
In the cases of Boolean and monotone convolutions, we can moreover prove that $\nu ^{\Box n} \in \p+$ implies $\nu \in \p+$, where $\Box$ is the Boolean 
or monotone convolution. This property was used to characterize the subordinators in terms of the L\'{e}vy-Khintchine representations \cite{Has2}. 
\end{remark}

Sometimes the limit distribution of Poisson's law of small numbers concerning 
a deformed convolution does not belong to $\p+$ \cite{KW}. We can prove a sufficient condition for this problem.  
We also show a condition for the central limit measure  to be contained in $\ps$.  

\begin{corollary}
$(1)$ We assume that $T(\p+) \subset \p+$. If Poisson's law of small numbers holds,  then
the limit distribution belongs to $\p+$. \\
$(2)$ We assume that $T(\ps) \subset \ps$. If the central limit theorem holds, then 
the limit distribution belongs to $\ps$. 
\end{corollary}
\begin{remark}
Poisson's law of small numbers and the central limit theorem mean the statements 
as in Theorem \ref{limit12}. 
\end{remark}
\begin{proof}
If we take $\mu^{(N)}:= (1 - \frac{\lambda}{N})\delta_0 + \frac{\lambda}{N}\delta_1 \in \p+$, then the limit distribution $\lim_{N \to \infty} (\mu^{(N)})^{\rt N} \in \p+$ by using Lemma \ref{closed1}. For the central limit theorem, we take 
$\mu:= \frac{1}{2}(\delta_{-1} + \delta_{1})$. Then $(D_{\frac{1}{\sqrt{N}}}\mu)^{\rt N} \in \ps$ and the limit distribution also belongs to $\ps$ again from Lemma \ref{closed1}.  
\end{proof}

\section{Transformations $V_{t,u,a}$} \label{Sec4}
We introduce a family of transformations denoted by $V_{t,u,a}$ and prove that the transformations satisfy the condition (\ref{asso}). Moreover, this family unifies 
the generalized $t$-transformation \cite{KY} and the $V_a$-transformation \cite{KW}.  

Let $f$: $\qa \to \real$, where $\qa$ is a subset of $\pa$.  
Typically $\qa$ is chosen to be $\pb$, $\pmo$ or $\pc$. 
Motivated by the generalized $t$-transformation in \cite{KY} and $V_a$-transformation in \cite{KW}, we look for a transform $V_{t,f}$ of the form 
\begin{equation}
\mu \mapsto V_{t,f}\mu, ~~ H_{V_{t,f} \mu} (z) = tH_\mu(z) + (1-t)z +  f(\mu).  
\end{equation} 
If $f(\mu) = (t-u)m(\mu)$, this is the same as the generalized $t$-transformation. 
If $t=1$ and $f(\mu) = a \sigma^2(\mu)$, this is the same as $V_a$-transformation. 

\begin{lemma}\label{asso2}
Assume that $\qa$ is closed under the convolution $\rhd_{\scriptscriptstyle V_{t,f}}$.
$V_{t,f}$ satisfies the associativity condition (\ref{asso}) considered in $\qa$ if and only if 
$f(\mu \rhd _{\scriptscriptstyle V_{t,f}} \nu) = f(\mu) + f(\nu)$ for all $\mu,\nu \in \qa$.
\end{lemma}
\begin{proof}
Denote $\rhd_{V_{t,f}}$ by $\rhd_{t,f}$ for simplicity. 
Applying (\ref{aa}) we obtain 
\begin{equation}
\begin{split}
H_{V_{t,f} (\mu \rhd_{t,f} \nu)}(z) 
             &= tH_{\mu \rhd_{t,f} \nu}(z) + (1-t)z + f( \mu \rhd_{t,f} \nu) \\
             &= tH_\mu \circ H_{V_{t,f} \nu}(z) + tH_\nu(z) - tH_{V_{t,f} \nu}(z) + (1-t)z  + f(\mu \rhd_{t,f} \nu) \\ 
             &= tH_\mu \circ H_{V_{t,f} \nu}(z) + tH_\nu(z) + (1-t)z + f(\nu) - tH_{V_{t,f} \nu}(z)  + f(\mu \rhd_{t,f} \nu) - f(\nu) \\
             &= tH_\mu \circ H_{V_{t,f} \nu}(z) + (1-t)H_{V_{t,f}\nu}(z) + f(\mu \rhd_{t,f} \nu) - f(\nu). \\
\end{split}
\end{equation}
On the other hand, we have
\begin{equation}
\begin{split}
H_{V_{t,f} \mu \rhd V_{t,f} \nu}(z) &= H_{V_{t,f} \mu} \circ H_{V_{t,f} \nu}(z) \\
                         &= tH_\mu \circ H_{V_{t,f} \nu}(z) + (1-t)H_{V_{t,f} \nu}(z) + f(\mu). 
\end{split}
\end{equation}
Therefore, the associativity condition (\ref{asso}) is equivalent to $f(\mu \rhd _{t,f} \nu) = f(\mu) + f(\nu)$.
\end{proof}

We define transformations $V_{t,u,a}$ by 
letting 
\begin{equation}
f(\mu) = f_{t,u,a}(\mu) = (t-u) m(\mu) + a\sigma^2(\mu).
\end{equation}
More clearly, we define
\begin{equation}\label{new}
H_{V_{t,u,a}\mu}(z) = tH_\mu(z) + (1-t)z +  (t-u) m(\mu) + a\sigma^2(\mu).
\end{equation} 
We expect that higher order moments for $f$ have nontrivial structure, but we do not treat them in this article. 
We use the notation $\rhd_{t,u,a}$ for the convolution defined by $V_{t,u,a}$. 

\begin{theorem}\label{gene}
The convolution $\rhd_{t,u,a}$ defined on $\pb$ is associative.  
\end{theorem}
\begin{proof}
This fact follows from Lemma \ref{asso2} and Proposition \ref{addi}. 
\end{proof}

In order to calculate the inverse transformation of $V_{t,u,a}$, we show the following facts. 
\begin{lemma} We have the following equalities. 
\begin{gather}
m(V_{t,u,a}\mu) = um(\mu) - a\sigma^2(\mu), \\
\sigma^2(V_{t,u,a}\mu) = t \sigma^2(\mu). 
\end{gather} 
\end{lemma} 
\begin{proof}
A direct computation leads to 
\begin{gather}
H_{V_{t,u,a}\mu}(z) = -um(\mu) + a\sigma^2(\mu) + z + t\int \frac{1}{x-z}d\rho_\mu (x), 
\end{gather}
from which  and Lemma \ref{maa} the conclusion follows. 
\end{proof}

\begin{proposition} We have the following equality:
\begin{equation}
V_{t',u',a'}V_{t,u,a} = V_{t't, u'u, u'a +a't} \text{~for~} t \geq 0, ~u , ~a \in \real. 
\end{equation}
In particular, we have  
\begin{equation}
V_{t,u,a,} ^{-1} = V_{t^{-1}, u^{-1}, -\frac{a}{tu}} \text{~for~} t > 0, ~u \neq 0, ~a \in \real. 
\end{equation}
\end{proposition}

\begin{proposition}\label{symV}
$(1)$ $\p+ \cap \pb$ is closed under the convolution $\rhd_{t,u,a}$ if and only if $u \geq t$ and $a = 0$.  \\
$(2)$ $\ps \cap \pb$ is closed under the convolution $\rhd_{t,u,a}$ if and only if $a = 0$. 
\end{proposition}
\begin{proof}
These facts are easy consequences of Proposition \ref{positive2}. 
\end{proof}
\begin{example} \normalfont 
\begin{itemize}
\item[(1)] Oravecz introduced the Fermi convolution $\bullet$ in \cite{O1}.  
He mentioned a relation between the Fermi convolution 
and the c-free convolution: 
\begin{equation}
(\mu \bullet \nu, \delta_{m(\mu)} \boxplus \delta_{m(\nu)}) = (\mu, \delta_{m(\mu)}) \boxplus (\nu, \delta_{m(\nu)}),
\end{equation}  
where $m(\mu)$ denotes the mean of $\mu$. We can easily extend the Fermi convolution to the convolution coming from the map $F_u$ defined by $F_u \mu = \delta_{um(\mu)}$. Clearly $V_{0,u,0} = F_u$. An associative convolution $\rhd_{F_u}$ 
arises from $F_u$:  
\begin{equation}\label{Fe}
(\mu \rhd_{F_u} \nu, \delta_{um(\mu) + um(\nu)}) = (\mu, \delta_{um(\mu)}) \rhd (\nu, \delta_{um(\nu)}). 
\end{equation} 
\item[(2)] The $t$-transformation is realized as $\mathcal{U}_t = V_{t, t, 0}$. An associative convolution $\rhd_t$ arises from the relation 
\begin{equation}\label{t-tra}
(\mu \rhd_{t} \nu, \mathcal{U}_t (\mu) \rhd \mathcal{U}_t(\nu)) = (\mu, \mathcal{U}_t (\mu)) \rhd (\nu, \mathcal{U}_t (\nu)). 
\end{equation} 
We note that the $t$-transformation interpolates the Boolean and monotone convolutions: they appear when $t=0$ and $t = 1$, respectively.    
\item[(3)] The $V_a$-transformation is equal to $V_{1,1,a}$. 
\end{itemize}
\end{example}

In the following we make the meaning of the results in this section clearer.  It is known that the $t$-transformation $\mathcal{U}_t$ ($t > 0$) satisfies the condition (\ref{Bo}), so that a new convolution $\boxplus_{\mathcal{U}_t}$ \cite{BW1} can be defined. This convolution can also be written as 
\begin{equation}\label{defo}
 \mu \boxplus_{\mathcal{U}_t} \nu = \mathcal{U}_{1 \slash t}(\mathcal{U}_t(\mu) \boxplus  \mathcal{U}_t(\nu)) 
\end{equation}
for $t > 0$. 
Apart from the context of the c-free convolution, it seems interesting to study 
the deformation of Boolean and tensor convolutions defined by the right hand side of (\ref{defo}), 
with $\boxplus$ replaced by $\uplus$ and $\ast$, respectively. 
The new convolutions were studied in \cite{BW2}.  
By definition, the deformed convolutions are associative and commutative. 

We can also define the same deformations in the monotone case.  
The results in this section show that the deformation has a natural meaning in terms of the c-monotone convolution  
as in the case of the free convolution (cf. (\ref{Bok}) and (\ref{Ha})).  
The above discussion is meaningful for any $T$ which is invertible such as some class of the $\mathbf{t}$-transformation.

\section{Deformations related to monotone infinitely divisible distributions}\label{Sec5}
Krystek and Wojakowski have introduced a deformation connected to a $\boxplus$-infinitely divisible distribution in \cite{KW}, which we explain now.  
For a $\boxplus$-infinitely divisible distribution $\varphi$ with a compact support, there corresponds 
a unique weakly continuous $\boxplus$-convolution semigroup $\{ \varphi_t \}_{t \geq 0}$ with $\varphi_0 = \delta_0$ and $\varphi_1 = \varphi$. 
Define a transformation $\Phi_t ^{\varphi}$ by 
\begin{equation}
\Phi_t^\varphi \mu = \varphi_{\sigma^2(\mu)t}.
\end{equation}
This map satisfies the condition (\ref{Bo}). 

We introduce the monotone analog of $\Phi_{t} ^{\varphi}$. 
\begin{definition}
For a $\rhd$-infinitely divisible distribution $\xi \in \pb$, let $\{ \xi_t \}_{t \geq 0}$ be the corresponding weakly continuous 
$\rhd$-convolution semigroup with $\xi_0 = \delta_0$ and $\xi_1 = \xi$.
Let $f$: $\pb \to \real$. We define a transformation $\Xi_{f} ^{\xi}$ by setting
\begin{equation}
\Xi_{f} ^{\xi}\mu := \xi_{f(\mu)}.  
\end{equation}
\end{definition}

\begin{lemma}
$\Xi_{f} ^{\xi}$ satisfies the associativity condition (\ref{asso}) in $\pb$ if and only if $f(\mu \rhd_{\Xi_f ^{\xi}} \nu) = f(\mu) + f(\nu)$ for all $\mu$, 
$\nu \in \pb$. 
\end{lemma}
\begin{proof}
This fact follows from the equality $\Xi_{f} ^{\xi}\mu \rhd \Xi_{f} ^{\xi}\nu = \xi_{f(\mu)+f(\nu)}$. 
\end{proof}
\begin{theorem}
The map $\Xi_{t} ^{\xi}$ $(t \geq 0)$ defined by $f(\mu)=  t\sigma^2(\mu)$ satisfies the condition (\ref{asso}).  
\end{theorem}
\begin{proof}
The fact follows from Proposition \ref{addi}. 
\end{proof}
\begin{remark}
$(1)$ 
If $\xi = \delta_a$, the map $f_{s,t}(\mu) = -s m(\mu) + t \sigma^2(\mu)$ $(s, t \in \real)$ is also possible. In this case we have $\Xi_{f_{u,a}}^{\delta_{1}} =  V_{0,u,a}$.  \\
$(2)$ Higher order moments may be  possible for $f$, which we do not consider in this paper. 
\end{remark} 

\begin{proposition}
$(1)$ $\p+ \cap \pb$ is closed under the convolution $\rhd_{\Xi_{t} ^{\xi}}$ if and only if $\xi \in \p+ \cap \pb$. \\ 
$(2)$ We assume that $\xi \in \pc$. Then $\ps \cap \pc$ is closed under the convolution $\rhd_{\Xi_{t} ^{\xi}}$ if and only if $\xi \in \ps \cap \pc$.
\end{proposition}
\begin{proof}
This is an immediate consequence of Lemma \ref{subordinator}, Lemma \ref{symm} and Proposition \ref{positive2}. 
\end{proof}

\section{Cumulants for a general convolution}\label{Cum}
Only in this section, $r_n(\mu)$ denote general cumulants, not only the monotone cumulants. 
 
To define cumulants for deformed convolutions in the section \ref{Cumulants}, 
we consider what are cumulants of a convolution product. We have clarified three axioms of cumulants in \cite{H-S} for random variables; 
however, we need more general axioms to treat convolutions appearing in Sections \ref{Sec4}, \ref{Sec5}. 
Results in this section are quite general and will be applicable to other convolutions which do not appear in this paper. 

Let $\Box$ be a convolution defined on $\pmo$. We shall treat convolutions which are not necessarily commutative for the later applications. All results in this section hold for both $\pmo$ and $\pc$
, except for Theorem \ref{add1}. Then we use the set $\pmo$ mainly. 
\begin{definition}
(1) We define recursively $\nu ^{\Box n} := \nu \Box \nu ^{\Box n-1}$ for $\nu \in \pmo$. 
$\Box$ is said to be power associative if $\nu^{\Box (n + m)} = \nu^{\Box n} \Box \nu^{\Box m}$ for all $m$, $n \geq 0$. 
\end{definition}

Let $m_n(\mu)$ be the $n$-th moment of $\mu \in \pmo$. 
We put the following assumptions.  
\begin{itemize}
\item[(M1)] There exists a universal polynomial $P_n$ of $2n - 2$ variables for each $n \geq 1$ such that 
\begin{equation}\label{mom1}
m_n (\mu \Box \nu) = m_n(\mu) + m_n(\nu) + P_n(m_1(\mu), \cdots, m_{n-1}(\mu), m_1(\nu), \cdots, m_{n-1}(\nu)). 
\end{equation}
\item[(M2)] The polynomial $P_n$ contains no constants for any $n \geq 1$.  
\end{itemize} 
\begin{remark}
The condition (M2) is equal to the condition $\delta_0 \Box \delta_0 = \delta_0$. 
\end{remark}

Let $r_n (\mu)$ be a polynomial of $\{ m_k(\mu) \}_{k \geq 1}$ for any $n \geq 1$. 
We consider the following properties. 
\begin{itemize} 
\item[(C1)] Power additivity: for any $n$, $N \geq 1$,   
\begin{equation}
r_n(\mu ^{\Box N}) = N r_n(\mu).
\end{equation}
\item[(C2)] There exists a universal polynomial $Q_n$ of $n-1$ variables such that 
\begin{equation}\label{cum1}
r_n(\mu) = m_n(\mu) + Q_n(m_1(\mu), \cdots, m_{n-1}(\mu)). 
\end{equation} 
\item[(C2')] In addition to the condition (C2), the polynomial $Q_n$ never contains linear terms $m_k(\mu)$, $1 \leq k \leq n-1$ for any $n$.  
\end{itemize}
$Q_1$ is understood to be a constant which turns out to be $0$ in Proposition \ref{trans2}. 

If a sequence $\{r_n \}$ satisfies (C2), we can write $m_n$ in terms of $r_n$ as 
\begin{equation}\label{mom2}
m_n(\mu) = r_n(\mu) + R_n(r_1(\mu), \cdots, r_{n-1}(\mu)), 
\end{equation}
where $R_n$ is a polynomial of $n-1$ variables. 

We note that in many important examples the condition of homogeneity \\ 
\begin{equation}\label{C2''}
r_n(D_{\lambda } \mu) = \lambda ^n r_n(\mu)  
\end{equation}
holds. Indeed, this condition holds for tensor, free, Boolean and monotone cumulants. Clearly (C2) and (\ref{C2''}) imply (C2').  
 We do not assume this condition since the uniqueness of cumulants follows from only (C1) and (C2') (see Proposition \ref{existence-uniqueness}). Moreover, 
there are examples which satisfy (C2') but do not satisfy (\ref{C2''}) such as cumulants for a convolution deformed by the $V_a$-transformation \cite{KW}.

If there exists a sequence $\{r_n \}$ satisfying (C1) and (C2), 
we consider a transformation of the form 
\begin{equation}\label{trans1}
r_n \mapsto r'_n := r_n + \sum_{k = 1} ^{n-1} a_{n, k} r_k 
\end{equation}
for real numbers $ a_{n, k}$, $1 \leq k \leq n-1$, $2 \leq n < \infty$.  
This transformation clearly preserves the properties (C1) and (C2). 
Moreover, we obtain the following property (1). 
\begin{proposition}\label{trans2}
We assume (M1) and (M2) and assume that $\Box$ is power associative. \\
(1) If there are two sequences $\{r_n \}$ and $\{r'_n \}$ satisfying (C1) and (C2), there exists a unique transformation of the form (\ref{trans1}) which maps $\{r_n \}$ to $\{r'_n \}$. \\
(2) The polynomial $Q_n$ in (C2) never contains a constant term. 
\end{proposition}
\begin{proof}
  (1) There exists a polynomial $A_n$ of variables $n-1$ for each $n \geq 1$ 
such that $r'_n = r_n + A_n(r_1, \cdots, r_{n-1})$ by using (\ref{cum1}) and (\ref{mom2}). Replacing $\mu$ by $\mu^{\Box N}$, we obtain $Nr'_n = Nr_n + A_n(Nr_1, \cdots, Nr_{n-1})$ for any $N$. This is an equality between polynomials of $N$, and hence, $A_n$ is of the form $A_n(r_1, \cdots, r_{n-1}) = \sum_{k=1}^{n-1} a_{n,k}r_k$. \\
(2)  We show the fact inductively. For $n = 1$, there exists $b_1 \in \real$ such that $r_1 = m_1 + b_1$. Since $P_n$ does not contain a constant term in (\ref{mom1}), we have $m_1(\mu \Box \mu) = 2 m_1(\mu)$, which implies $2r_1(\mu) - b_1 = 2 r_1(\mu) - 2b_1$. Therefore, $b_1 = 0$. 
We assume that $Q_n$ does not contain a constant term for $n \leq k$. 
Using a similar argument, we can prove that $Q_{k+1}$ does not contain a constant term.  
\end{proof}

\begin{proposition}\label{existence-uniqueness}
We assume (M1) and (M2) and assume that $\Box$ is power associative. 
 The following statements are equivalent: 
\begin{itemize}
\item[(a)] There exists a sequence $\{r_n \}_{n \geq 1}$ satisfying (C1) and (C2); 
\item[(b)] There exists a sequence $\{r_n \}_{n \geq 1}$ satisfying (C1) and (C2'); 
\item[(c)] $m_n(\mu^{\Box N})$ is a polynomial of $m_1(\mu), \cdots, m_n(\mu)$ and $N$ for any $n$. 
\end{itemize}
Moreover, the sequence $\{r_n \}$ in (b) is unique and is given by 
\begin{equation}\label{cum7}
r_n(\mu) = \frac{\partial }{\partial N} m_n(\mu ^{\Box N}) \Big{|}_{N = 0}. 
\end{equation} 
\end{proposition}
\begin{remark}
We can see from (\ref{cum7}) that cumulants are strongly related to a convolution semigroup $\{\mu_t \}_{t \geq 0}$ with $\mu_0 = \delta_0$ and to infinite divisibility.  
\end{remark}
\begin{proof}
 (a) $\Rightarrow$ (c): if there exists a sequence $\{r_n \}_{n \geq 1}$ satisfying (C1) and (C2), 
we have 
\begin{equation}\label{cum5}
\begin{split}
m_n(\mu ^{\Box N}) &= r_n(\mu ^{\Box N}) + R_n(r_1(\mu ^{\Box N}, \cdots, r_{n-1}(\mu^{\Box N})) \\
           &= Nr_n(\mu) + R_n(Nr_1(\mu), \cdots, Nr_{n-1}(\mu)) \\
           &= Nm_n(\mu) + NQ_n(m_1(\mu), \cdots, m_{n-1}(\mu)) \\ &~~~~+ R_n(Nm_1(\mu), \cdots, Nm_{n-1}(\mu) + NQ_{n-1}(m_1(\mu), \cdots, m_{n-2}(\mu))). 
\end{split}
\end{equation}
Therefore, $m_n(\mu^{\Box N})$ is a polynomial of $N$ and $m_k(\mu)$. 

(c) $\Rightarrow$ (a):  by using (M1), (M2) and the assumption (c), $m_n(\mu ^{\Box N})$ has such a form as 
\begin{equation}\label{cum4}
m_n(\mu ^{\Box N}) = Nm_n(\mu) + \sum_{l = 0} ^{L}N^l S_l (m_1(\mu), \cdots, m_{n-1}(\mu)) 
\end{equation} 
for polynomials $S_l$ and an $L \in \nat$. 
We define 
\begin{equation}\label{cum6}
\begin{split}
r_n(\mu) &:= \frac{\partial }{\partial N} m_n(\mu ^{\Box N}) \Big{|}_{N = 0} \\
         &= m_n (\mu) + S_1(m_1(\mu), \cdots, m_{n-1}(\mu)).  
\end{split}
\end{equation} 
The power associativity of $\Box$ implies (C1).  (C2) follows from (\ref{cum6}). 

(a) $\Rightarrow$ (b): for a sequence $\{r_n \}$ satisfying (C1) and (C2), we can write $r_n$ in the form $r_n = m_n + \sum_{k = 1}^{n-1}b_{n,k}m_k + T_n(m_1, \cdots, m_{n-1})$, where $T_n$ is a polynomial which does not contain  
linear terms $m_k$, $1 \leq k \leq n-1$. We define a new sequence $\{r'_n \}$ inductively as follows: $r'_1 := r_1$, 
$r'_2  = r_2 - b_{2,1}r_1$, $r'_n = r_n - \sum_{k = 1} ^{n-1}a_{n, k}r'_k$ for $n \geq 2$. Then $r'_n$ do not contain linear terms $m_k$. 

We note that $Q_n$ does not contain a constant term from Proposition \ref{trans2} (2). If there exists a sequence $\{r_n \}$ satisfying (C1) and (C2'),  
the corresponding polynomial $R_n$ in (\ref{mom2}) also does not contain 
linear terms $m_k$, $1 \leq k \leq n-1$ or a constant term. Therefore, the equality 
$m_n(\mu ^{\Box N}) = Nr_n(\mu) + R_n(Nr_1(\mu), \cdots, Nr_{n-1}(\mu))$ implies that $r_n = \frac{\partial }{\partial N} m_n(\mu ^{\Box N}) \Big{|}_{N = 0}$. 
\end{proof}

\begin{definition}\label{cumulants111}
Let $\Box$ be a power associative convolution defined on $\pmo$ satisfying 
(M1) and (M2). Then the polynomials $r_n$ satisfying (C1) and (C2') 
are called the cumulants for the convolution $\Box$. Cumulants are unique. 
\end{definition}
\begin{remark}
This definition extends the cumulants for the tensor, free, Boolean and monotone convolutions. 
\end{remark}
We can prove the existence of cumulants.   
\begin{theorem}
We assume the conditions (M1) and (M2) for a power associative convolution $\Box$. Then cumulants of $\Box$ exist. 
\end{theorem}
\begin{proof}It is sufficient to prove that $m_n(\mu^{\Box N})$ is a polynomial of $N$ due to Proposition \ref{existence-uniqueness}. Then the proof is the same as in \cite{H-S}, which we omit here. 
\end{proof}

We discuss when the additivity of cumulants holds. In the proof of the following theorem, we assume the convolution is defined on $\pc$ so that moments determine a unique probability measure. 
\begin{theorem}\label{add1}
Let $\Box$ be a power associative convolution defined on $\pc$ satisfying (M1) and (M2). Let $r_n$ be the cumulants. Then the following conditions are equivalent. 
\begin{itemize}
\item[(1)] $r_n(\mu \Box \nu) = r_n(\mu) + r_n(\nu)$ for all $n$ and $\mu$,  $\nu \in \pc$. 
\item[(2)] $\Box$ is associative and commutative, and moreover, $P_n$ in (\ref{mom1}) does not contain linear terms $m_k(\mu)$ or $m_k(\nu)$, $1 \leq k \leq n-1$. 
\end{itemize}
\end{theorem}
\begin{proof}
(1) $\Rightarrow$ (2): the associativity and commutativity follow immediately since a probability measure with compact support is determined by the cumulants.   
From  (M1) and (\ref{mom2}) we obtain the identity 
\begin{equation*}
\begin{split}
&P_n(m_1(\mu), \cdots, m_{n-1}(\mu), m_1(\nu), \cdots, m_{n-1}(\nu)) \\
&~~~~~~=Q_n(m_1(\mu), \cdots, m_{n-1}(\mu)) + Q_n(m_1(\nu), \cdots, m_{n-1}(\nu)) \\
&~~~~~~~~~~~~~~~~+ R_n(r_1(\mu) + r_1(\nu), \cdots, r_{n-1}(\mu) + r_{n-1}(\nu)). 
\end{split}
\end{equation*}
It follows from  (C2') that $Q_n$ and $R_n$ do not contain linear terms. 

(2) $\Rightarrow$ (1): Using  (M1), (C2') and (\ref{mom2})  we have 
\begin{equation*}
\begin{split}
r_n(\mu \Box \nu) &= m_n(\mu \Box \nu) + Q_n(m_1(\mu \Box \nu), \cdots, m_{n-1}(\mu \Box \nu)) \\
           &= r_n(\mu) + r_n(\nu) + P_n(m_1(\mu), \cdots, m_{n-1}(\mu), m_1(\nu), \cdots, m_{n-1}(\nu)) \\ 
&~~~~~~~~~~~+ R_n(r_1(\mu), \cdots, r_{n-1}(\mu)) + R_n(r_1(\nu), \cdots, r_{n-1}(\nu)) \\
&~~~~~~~~~~~~~~~~~+ Q_n(m_1(\mu \Box \nu), \cdots, m_{n-1}(\mu \Box \nu)).   
\end{split}
\end{equation*}
Therefore, there exists a polynomial $U_n$ which does not contain linear terms such that $r_n(\mu \Box \nu) = r_n(\mu) + r_n(\nu) + U_n(r_1(\mu), \cdots, r_{n-1}(\mu), r_1(\nu), \cdots, r_{n-1}(\nu))$. We replace $\mu$ and $\nu$ by $\mu^{\Box N}$ and $\nu^{\Box N}$, respectively. The associativity and commutativity implies that 
$r_n(\mu^{\Box N} \Box \nu ^{\Box N}) = r_n((\mu \Box \nu)^{\Box N}) = Nr_n(\mu \Box \nu)$.  Then $Nr_n(\mu \Box \nu) = N r_n(\mu) + N r_n(\nu) + U_n(Nr_1(\mu), \cdots, Nr_{n-1}(\mu), Nr_1(\nu), \cdots, Nr_{n-1}(\nu))$. This can be seen as an identity between polynomials of $N$; therefore, we have $U_n = 0$. 
\end{proof}

Limit theorems can be formulated in terms of moments and cumulants. The proofs are easy. 
\begin{theorem}\label{centrallimit123}
 Let $\Box$ be a power associative convolution defined on $\pmo$ satisfying (M1) and (M2). Let $r_n$ be the cumulants. \\
(1) (Central limit theorem) For $\mu \in \pmo$ with $m_1(\mu) = 0$ and $m_2(\mu) = 1$, we define $\mu_N := (D_{\frac{1}{\sqrt{N}}} \mu) ^{\Box N}$. Then 
$r_1(\mu_N) \to 0$, $r_2(\mu_N) \to 1$ and $r_n(\mu_N) \to 0$ as $N \to \infty$ for any $n \geq 3$.\\
(2) (Poisson's law of small numbers)  Let $\{\mu^{(N)} \}$ be a sequence such that for any $n \geq 1$ $Nm_n(\mu^{(N)}) \to \lambda > 0$ as $N \to \infty$. We define $\mu_N := (\mu ^{(N)})^{\Box N}$. Then $r_n(\mu_N) \to \lambda$ as $N \to \infty$ for any $n \geq 1$. 
\end{theorem}

\section{Cumulants for deformed convolutions}\label{Cumulants}
We define $r_n ^T(\mu):= r_n(\mu, T\mu)$.  $r_n ^T(\mu)$ turn out to be cumulants for the convolution $\rt$ in the sense of Definition \ref{cumulants111}.  
\begin{proposition}
We assume that there exists a polynomial $V_n$ of $n+1$ variables, which does not contain a constant term, such that 
\begin{equation}\label{cum12}
m_n(T\mu) = V_n(m_1(\mu), \cdots, m_{n+1}(\mu))
\end{equation}
for any $n \geq 1$. Then the conditions (M1) and (M2) hold for the convolution $\rt$. 
\end{proposition}
\begin{proof}
(M1) follows from (\ref{cum14}) and (\ref{cum12});  (M2) follows from the fact that both $W_n$ and $Y_n$ do not contain constant terms nor linear terms.  
\end{proof}
\begin{theorem}\label{cum34}
Let $T$: $\pmo \to \pmo$ be a map satisfying (\ref{asso}) and (\ref{cum12}). Then $r_n(\mu, T\mu)$ satisfy the conditions (C1) and (C2') for the convolution $\rt$. 
\end{theorem}
\begin{proof}
(C2') follows from the definition of c-monotone cumulants and (\ref{cum12}). (C1) can be proved showed as follows: $r_n^T (\mu^{\rt N}) = r_n(\mu^{\rt N}, T(\mu^{\rt N})) = r_n((\mu, T\mu)^{\rhd N}) = Nr_n(\mu, T\mu) = Nr_n ^T(\mu)$.  
\end{proof}
The c-free cumulants $R_n(\mu, T\mu)$ satisfy the conditions (C1) and (C2') 
under similar conditions.   
\begin{proposition}
Let $T$: $\pmo \to \pmo$ be a map satisfying the condition (\ref{Bo}). 
We assume that the $n$-th moment of $T\mu$ is of the form 
\begin{equation}\label{Tmom1}
m_n(T\mu) = V_n(m_1(\mu), \cdots, m_{n+1}(\mu))  
\end{equation}
for any $n \geq 1$, where $V_n$ is a polynomial which does not contain a constant term. 
Then the convolution $\boxplus_T$ satisfies the conditions (M1) and (M2), and $R_n(\mu, T\mu)$ 
satisfies the conditions (C1) and (C2').    
\end{proposition}
\begin{remark}
All the convolutions studied in \cite{BW1,BW2,KW,KY,O1} satisfy the condition (\ref{Tmom1}). 
\end{remark}

\section{Limit theorems for deformed convolutions}\label{Lim2}
We can apply Theorem \ref{centrallimit123} to the convolution $\rt$ under the conditions (\ref{asso}) and (\ref{cum12}). 
We summarize the statements combining Theorem \ref{centrallimit123} and Theorem \ref{cum34}. 

\begin{theorem}\label{limit11}Let $T$: $\pmo \to \pmo$ be a map which satisfies (\ref{asso}) and (\ref{cum12}). \\
(1) (Central limit theorem) Let $\mu$ be a probability measure in $\pmo$ with mean $0$ and variance $1$. We define 
$\mu_N := \big{(}D_{\frac{1}{\sqrt{N}}}\mu \big{)}^{\rt N}$.  
Then $m_n(\mu_N)$ converges to $m_n(\nu^{(T)}_1)$, where $m_n(\nu^{(T)}_t)$ are characterized by 
\begin{equation}\label{centrallim}
\frac{\partial}{\partial t} H_{\nu^{(T)}_t}(z) = -\frac{1}{H_{T\nu^{(T)}_t}(z)}.\end{equation}  
(2) (Poisson's law of small numbers) Let $\{ \mu^{(N)} \}_{N=1} ^{\infty}$ be a sequence of probability measures in $\pmo$ such that $N m_n(\mu^{(N)}) \to \lambda > 0$ as $N \to \infty$ for all $n \geq 1$. 
We define $\mu_N := (\mu^{(N)})^{\rt N}$. Then $m_n(\mu_N)$ converges to $m_n(p^{(T)}_{\lambda})$, where  
 $m_n(p^{(T)}_{\lambda})$ are characterized by 
\begin{equation}\label{Poissonlim}
\frac{\partial}{\partial \lambda} H_{p^{(T)}_{\lambda}}(z) = \frac{H_{Tp^{(T)}_{\lambda}}(z)}{1 - H_{Tp^{(T)}_{\lambda}}(z)}. 
\end{equation}
\end{theorem} 

In this section we calculate the limit distributions for $T$ constructed in Sections \ref{Sec4} and \ref{Sec5}. 
If $T$ is invertible, we can use monotone cumulants to calculate the limit distributions since $\mu \rt \nu = T^{-1}(T\mu \rhd T\nu)$. Cumulants introduced in Section \ref{Cumulants}, however, enable us to calculate the limit distributions for even non-invertible $T$. In this section, we always use cumulants introduced in Section \ref{Cumulants}. 

\subsection{Transformations $V_{t,u,a}$}
We now calculate the central limit measure for the convolution $\rhd_{t, u, a}$.  We only calculate the two cases $a = 0$ and $t = 0$; otherwise explicit expressions of  the 
limit measures are difficult. For simlicity, let $r^{(t,u,a)}_n(\mu)$ denote the cumulants $r^{V_{t,u,a}}_n(\mu)$. In this section we use two logarithms $\log_{[1]}$ and $\log_{[2]}$:  $\log_{[1]}(z)$ is defined by $\log_{[1]}(z) := \log |z| + i\arg (z)$,   
$\arg (z) \in (-\pi, \pi)$, $z \in \comp \backslash (-\infty, 0]$;  
$\log_{[2]}$ is defined by $\log |z| + i\arg (z)$, $\arg (z) \in (0, 2\pi)$, 
$z \in \comp \backslash [0, \infty)$. 
Let $\sqrt{z}$ be $\exp(\frac{1}{2}\log_{[2]} z)$ for $z \in \comp \backslash [0, \infty)$. Then, for instance,  the Cauchy transform of the normalized arcsine law becomes  $\frac{1}{\sqrt{z^2 - 2}}$ for $z \in \com+$. 
 
\begin{theorem}\label{centrallim2}
(1) Let $\mu$ be a probability measure in $\pmo$ with mean $0$ and variance $1$. Then 
$\mu_N := (D_{\frac{1}{\sqrt{N}}}\mu)^{\rhd_{t,u,0} N}$ converges weakly to a Kesten distribution $\nu^{(t,0)}$. 
The absolutely continuous part  is  $\frac{1}{2\pi}\frac{\sqrt{2t -x^2}}{1 - (1 - \frac{t}{2})x^2}dx$ on $[-\sqrt{2t}, \sqrt{2t}]$. There is no singular part for $t \geq 1$, 
but $\nu^{(t,0)}$ contains atoms at $x = \pm \frac{1}{\sqrt{1 - \frac{t}{2}}}$ for $t < 1$. \\
(2) Let $\mu$ be a probability measure in $\pmo$ with mean $0$ and variance $1$. Then 
$\mu_N := (D_{\frac{1}{\sqrt{N}}}\mu)^{\rhd_{0,0,a} N}$ converges weakly to a probability measure $\nu^{(0,a)}$. The absolutely continuous part of $\nu^{(0,a)}$ is given by  
\begin{equation}
\begin{split}
\nu^{(0,a)}|_{ac} = \begin{cases}
&\frac{a}{(\log|1 + \frac{a}{x}| - ax)^2 + \pi^2} dx,~~x \in [-a, 0],~~a > 0, \\  
&\frac{|a|}{(\log|1 + \frac{a}{x}| - ax)^2 + \pi^2} dx,~~x \in [0, |a|],~~a < 0.
\end{cases}\end{split}
\end{equation}
 $\nu^{(0,a)}$ contains two atoms: one in $(-\infty, -a)$ and the other in $(0, \infty)$ if $a > 0$; one in $(-\infty, 0)$ and the other in $(|a|, \infty)$ if $a < 0$.   
\end{theorem}
\begin{remark}
(1) Kesten distributions also appear in the central limit theorem of $\boxplus_{\mathcal{U}_t}$ \cite{BW1,BW2} with the parameter $t$ replaced by $2t$.   \\
(2) The limit distribution of (2) is symmetric only in the case of $a = 0$ where the convolution becomes a Boolean convolution (cf. Proposition \ref{symV}).  
\end{remark}
\begin{proof}
(1) Let $\{\nu^{(t, 0)}_s \}_{s \geq 0}$ be a (formal) convolution semigroup which is a solution of (\ref{centrallim}) for $T = V_{t,u,0}$. (The word ``formal'' means that the limit moments might not be deterministic. Therefore, we consider $\nu^{(t,0)}_s$ as a sequence of moments.) We note that $m_1(\nu^{(t, 0)}_s) = sr^{(t, u, 0)} _1(\nu^{(t, 0)}_1) = 0$. Then $H_{V_{t,u,a}\nu^{(t, a)}_s}(z) = tH_{\nu^{(t, a)}_s}(z) + (1 - t)z$. We let $H_s(z)$ denote $H_{\nu^{(t, 0)}_s}(z)$ for simplicity. 
(\ref{centrallim}) can be integrated and we obtain 
$\frac{t}{2}H_s(z)^2 + (1-t)z H_s(z) = -s +(1- \frac{t}{2})z^2$, which implies 
\begin{equation}
G_s(z) = \frac{(\frac{1}{2} - \frac{t}{2}) + \frac{1}{2}\sqrt{z^2 - 2st}}{(1 - \frac{t}{2})z^2 - s}. 
\end{equation} 
$G_1$ is the Cauchy transform of a Kesten distribution (see \cite{BW2}), whose support is compact. 
Then the weak convergence holds (see Theorem 4.5.5 of \cite{KLC}). \\
(2) Let $\{\nu^{(0, a)}_s \}_{s \geq 0}$ be a (formal) convolution semigroup which is a solution of (\ref{centrallim}) for $T = V_{0,0,a}$. We note that $m_1(\nu^{(0, a)}_s) = 0$ and $\sigma^2(\nu^{(0, a)}_s) = r^{(0, 0, a)}_2(\nu^{(0, a)}_s) =  s$. Then $H_{V_{0,0,a}\nu^{(0, a)}_s}(z) = z + as$. Let $H_s(z)$ denote $H_{\nu^{(0, a)}_s}(z)$  for simplicity. 
We have 
\begin{equation}
\begin{split}
  H_s(z) &= -\int_0 ^s \frac{1}{z + ar} dr + z \\
         &= z - \frac{1}{a}\log_{[1]}\Big{(}1 + \frac{as}{z}\Big{)}. 
\end{split}
\end{equation}
\textbf{Case $a > 0$}: the absolutely continuous part of the limit distribution is $\frac{a}{(\log|1 + \frac{a}{x}| - ax)^2 + \pi^2} dx$ supported on the interval $\{x \in \real: G_1(x + i0) < 0 \} = [-a, 0]$. We can show that the limit distribution contains an atom in $(0, \infty)$ and the other in $(-\infty, -a)$. \\
\textbf{Case $a < 0$}: the absolutely continuous part is $\frac{|a|}{(\log|1 + \frac{a}{x}| - ax)^2 + \pi^2} dx$ supported on the interval  $[0, |a|]$. 
We can show that the limit distribution contains an atom in $(|a|, \infty)$ and the other in $(-\infty, 0)$. 

We note that the case $a = 0$ corresponds to the Boolean convolution, and hence, the limit distribution is $\frac{1}{2}(\delta_{-1} + \delta_1)$.    
\end{proof}

We calculate the limit distribution for Poisson's law of small numbers. We consider only 
the case $T = V_{0,u,a}$; otherwise, the explicit form is difficult to obtain.  
\begin{theorem}
Let $\{ \mu^{(N)} \}_{N=1} ^{\infty}$ be a sequence of probability measures in $\pmo$ such that $N m_n(\mu^{(N)}) \to \lambda > 0$ as $N \to \infty$ for all $n \geq 1$. 
Then $\mu_N := (\mu^{(N)})^{\rhd_{0,u,a} N}$ converges weakly to a compactly supported distribution $p^{(u,a)}_{\lambda}$. The absolutely continuous part of $p^{(u,a)}_{\lambda}$ is given by 
\begin{equation}
\begin{split}
p^{(u,a)}_{\lambda}|_{ac} = \begin{cases}
&\frac{a - u}{\big{(}\log|1 + \frac{(a-u)\lambda}{x-1}| - (a-u)(x - \lambda)\big{)}^2 + \pi^2} dx,~~x \in [1- (a-u)\lambda, 1],~~a > u, \\ 
&\frac{|a-u|}{\big{(}\log|1 + \frac{(a-u)\lambda}{x-1}| - (a-u)(x-\lambda)\big{)}^2 + \pi^2} dx,~~x \in [1, 1 + (u-a)\lambda],~~a < u.
\end{cases}\end{split}
\end{equation}
$p^{(u,a)}_{\lambda}$ contains two atoms: one in $(-\infty, 1- (a-u)\lambda)$ and the other in $(1, \infty)$ for $a > u$; one in $(-\infty, 1)$ and the other in $(1 + (u-a)\lambda, \infty)$ for $a < u$.   
\end{theorem}
\begin{remark}
One can see that $p^{(u,a)}_{\lambda}$ is in $\p+$ if and only if $u \geq a $ (cf. Proposition \ref{symV}).  
\end{remark}
\begin{proof}
We note that $m_1(p^{(u,a)}_{\lambda}) = r^{(0, u, a)}_1(p^{(u,a)}_{\lambda}) = \lambda$ and $\sigma^2(p^{(u,a)}_{\lambda}) = \lambda$. Then we obtain the differential equation 
\begin{equation}
\frac{\partial}{\partial \lambda} H_{p^{(u, a)}_{\lambda}}(z) = -1 - \frac{1}{z - 1 + (a - u)\lambda}. 
\end{equation}
The remaining arguments are similar to Theorem \ref{centrallim2} and we omit the proof. 
\end{proof}

\subsection{Deformations related to $\rhd$-infinitely divisible distributions} 
For a compactly supported $\rhd$-infinitely divisible distribution $\xi$, let $\{ \xi_t \}_{t \geq 0}$ be the corresponding weakly continuous $\rhd$-convolution semigroup with $\xi_0 = \delta_0$ and $\xi_1 = \xi$. 
Then $\xi_t$ is compactly supported for every $t > 0$ \cite{Mur3}. 
Let $\{ \nu^{[\xi,t]}_s \}_{s \geq 0}$ and $\{p^{[\xi, t]}_{s} \}_{s \geq 0}$ be the (formal) convolution semigroups 
defined by (\ref{centrallim}) and (\ref{Poissonlim}), respectively. 
Let $r^{[\xi, t]}_n(\mu)$ denote $r^{\Xi^\xi _t}_n(\mu)$.  
Since $r^{[\xi, t]}_2(\nu^{[\xi,t]}_s) = \sigma^2(\nu^{[\xi,t]}_s) = s$ we obtain $\Xi_{t} ^{\xi} (\nu^{[\xi,t]}_s) = \xi_{st}$. Similarly, we obtain $\Xi_{t} ^{\xi} (p^{[\xi,t]}_\lambda) = \xi_{\lambda t}$. Therefore, (\ref{centrallim}) and (\ref{Poissonlim}) become 
\begin{gather}
\frac{\partial}{\partial s} H_{\nu^{[\xi,t]}_s}(z) = -\frac{1}{H_{\xi_{st}}(z)}, \label{xi1}\\ 
\frac{\partial}{\partial \lambda} H_{p^{[\xi, t]}_{\lambda}}(z) = \frac{H_{\xi_{t\lambda}}(z)}{1 - H_{\xi_{t\lambda}}(z)}. \label{xi2}
\end{gather}
These equations have been defined in the sense of formal power series. However, once equations (\ref{xi1}) and (\ref{xi2}) are understood to be ordinary differential equations, 
the solutions are analytic outside a ball for every $s > 0$ and $\lambda > 0$. As a result, (\ref{xi1}) and (\ref{xi2}) give moments of compactly supported probability measures for each $s > 0$ and $\lambda$. Therefore, $\nu^{[\xi,t]}_s$ and $p^{[\xi, t]}_{\lambda}$ make sense as uniquely determined probability measures. Moreover, the convergence of moments in Theorem \ref{limit11} becomes the weak convergence. We summarize the above arguments. 
Let $\rhd_{[\xi, t]}$ denote $\rhd_{\Xi^\xi_t}$. 
\begin{theorem}\label{limit12} Let $\xi$ be a $\rhd$-infinitely divisible distribution in $\pc$. \\
(1) (Central limit theorem) Let $\mu$ be a probability measure in $\pmo$ with mean $0$ and variance $1$. Then 
$\mu_N := \big{(}D_{\frac{1}{\sqrt{N}}}\mu \big{)}^{\rhd_{[\xi,t]} N}$ converges to $\nu^{[\xi, t]}_1$ weakly.  \\ 
(2) (Poisson's law of small numbers) 
Let $\{ \mu^{(N)} \}_{N=1} ^{\infty}$ be a sequence of probability measures in $\pmo$ such that $N m_n(\mu^{(N)}) \to \lambda > 0$ as $N \to \infty$ for all $n \geq 1$. 
Then $\mu_N := (\mu^{(N)})^{\rhd_{[\xi, t]} N}$ converges to $p^{[\xi, t]}_{\lambda}$ weakly. 
\end{theorem} 

We calculate the limit distributions explicitly when $\xi$ is the normalized arcsine law. 

\begin{theorem} Let $\eta$ be the normalized arcsine law. \\
(1) The limit distribution  $\nu^{[\eta, t]}_1$ is the Kesten distribution $\nu^{(t,0)}$. \\
(2) The absolutely continuous part of $p^{[\eta, t]}_{\lambda}$ is supported on 
$[-\sqrt{2\lambda t}, \sqrt{2\lambda t}] \cup [1, \sqrt{2\lambda t+1}]$. The singular part consists of atoms: an atom exists in $(\sqrt{2\lambda t+1}, \infty)$; another exists in $(\sqrt{2\lambda t}, 1 )$ if $0 < t < \frac{1}{2\lambda}$ and $(1 - \frac{1}{t})\sqrt{2\lambda t} - \frac{1}{t} \log (1 - \sqrt{2\lambda t}) - \lambda < 0$; the other exists in $(-\infty, -\sqrt{2\lambda t})$ if $(\frac{1}{t} - 1)\sqrt{2\lambda t} - \frac{1}{t}\log(\sqrt{2\lambda t} + 1) > 0$. 
\end{theorem}
\begin{remark}
It is remarkable that the limit distribution in (1) also appears in Theorem 10 of \cite{KW} with the parameter $t$ replaced by $2t$.   
\end{remark}
\begin{proof}
(1) The differential equation (\ref{xi1}) becomes 
\begin{equation}
\frac{\partial}{\partial s} H_{\nu^{[\eta,t]}_s}(z) = -\frac{1}{\sqrt{z^2 - 2ts}},
\end{equation}
which implies $H_{\nu^{[\eta,t]}_s}(z) = (1- \frac{1}{t})z + \frac{1}{t}\sqrt{z^2 - 2ts}$. Therefore, the limit distribution is the Kesten distribution. \\
(2) The differential equation (\ref{xi2}) becomes 
\begin{equation}
\frac{\partial}{\partial \lambda} H_{p^{[\xi, t]}_{\lambda}}(z) = \frac{\sqrt{z^2 - 2 \lambda t}}{1 - \sqrt{z^2 - 2 \lambda t}}.  
\end{equation}
We can solve this and obtain 
\begin{equation}
H_{p^{[\xi, t]}_{\lambda}}(z) = \Big{(}1 -\frac{1}{t}\Big{)}z + \frac{1}{t}\sqrt{z^2 - 2\lambda t} + \frac{1}{t}\log_{[1]}\Big{(}\frac{\sqrt{z^2 - 2 \lambda t} - 1}{z - 1} \Big{)} - \lambda. 
\end{equation}
One can see that $\lim_{\im z \searrow 0}H_{p^{[\xi, t]}_{\lambda}}(z) > 0$ 
if and only if $\re z \in (-\sqrt{2\lambda t}, \sqrt{2\lambda t}) \cup [1, \sqrt{2\lambda t+1}]$. We remark that $H_{p^{[\xi, t]}_{\lambda}}(x)$ is strictly increasing
 in the intervals $(-\infty, -\sqrt{2\lambda t})$,  $(\sqrt{2\lambda t}, 1)$ and $(\sqrt{2\lambda t+1}, \infty)$. Then it is not difficult to show the existence of atoms. 
\end{proof}

\section*{Acknowledgements} 
This paper owes much to the joint work with Mr. Hayato Saigo on cumulants. The author would like to thank Mr. Hayato Saigo for many discussions about quantum probability, independence, umbral calculus and in particular, cumulants. He is grateful to Professor Izumi Ojima for reading the manuscript, suggesting improvements of many sentences and discussions about independence. He thanks Professor Marek Bo\.{z}ejko for guiding him to the notion of conditionally free independence and an important reference \cite{Fra}. He also thanks Professor Uwe Franz for fruitful discussions and for giving a seminar on the categorical treatment of independence during the visit to Kyoto. He also thanks Professor Shogo Tanimura, Mr. Ryo Harada, Mr. Hiroshi Ando and Mr. Kazuya Okamura for their comments and encouragement. This work was supported by Japan Society for the Promotion of Science, KAKENHI 21-5106. The author also thanks the support by Global COE Program at Kyoto University.

\end{document}